\documentclass[12pt,reqno]{article}

\usepackage[usenames]{color}
\usepackage{amssymb}
\usepackage{graphicx}
\usepackage{epsfig}
\usepackage{epstopdf}
\usepackage{amscd}
\usepackage{caption}

\usepackage[colorlinks=true,
linkcolor=webgreen,
filecolor=webbrown,
citecolor=webgreen]{hyperref}

\definecolor{webgreen}{rgb}{0,.5,0}
\definecolor{webbrown}{rgb}{.6,0,0}

\usepackage{color}
\usepackage{fullpage}
\usepackage{float}

\usepackage{graphics,amsmath,amssymb}
\usepackage{amsthm}
\usepackage{amsfonts}
\usepackage{latexsym}

\newcommand{\seqnum}[1]{\href{http://oeis.org/#1}{\underline{#1}}}

\begin{document}

\begin{center}
\epsfxsize=4in

\end{center}

\theoremstyle{plain}
\newtheorem{theorem}{Theorem}
\newtheorem{corollary}[theorem]{Corollary}
\newtheorem{lemma}[theorem]{Lemma}
\newtheorem{proposition}[theorem]{Proposition}
\newtheorem{question}[theorem]{Question}

\theoremstyle{definition}
\newtheorem{definition}[theorem]{Definition}
\newtheorem{example}[theorem]{Example}
\newtheorem{conjecture}[theorem]{Conjecture}

\theoremstyle{remark}
\newtheorem{remark}[theorem]{Remark}

\newcommand\BD{\mathrm{B}}
\newcommand\SD{\mathrm{S}}

\begin{center}
\vskip 1cm{\LARGE\bf 
Stern Sequences for a Family of Multidimensional Continued Fractions:  \\
\vskip .1in 
TRIP-Stern Sequences 
}
\vskip 1cm
\large
Ilya Amburg\\
Center for Applied Mathematics\\
Cornell University\\
Ithaca, NY 14853\\
USA \\
\href{mailto:ia244@cornell.edu}{\tt ia244@cornell.edu} \\
\ \\
Krishna Dasaratha\\
Department of Mathematics\\
Stanford University\\
Stanford, CA 94305\\
USA\\
\href{mailto:kdasarat@stanford.edu}{\tt kdasarat@stanford.edu}\\
\ \\
Laure Flapan\\
Department of Mathematics\\
University of California, Los Angeles\\
Los Angeles, CA 90095\\
USA \\
\href{mailto:lflapan@math.ucla.edu}{\tt lflapan@math.ucla.edu }\\
\ \\
Thomas Garrity\\
Department of Mathematics and Statistics\\
Williams College\\
Williamstown, MA 01267\\
USA \\
\href{mailto:tgarrity@williams.edu}{\tt tgarrity@williams.edu}\\
\ \\
Chansoo Lee\\
Department of Computer Science\\
University of Michigan, Ann Arbor\\
Ann Arbor, MI 48109\\
USA \\
\href{mailto:chansool@umich.edu}{\tt chansool@umich.edu}  \\
\ \\
Cornelia Mihaila           \\
Department of Mathematics\\
University of Texas, Austin\\
Austin, TX 78712\\
USA \\
\href{mailto:cmihaila@math.utexas.edu}{\tt cmihaila@math.utexas.edu} \\
\ \\
Nicholas Neumann-Chun          \\
Department of Mathematics and Statistics\\
Williams College\\
Williamstown, MA 01267\\
USA \\
\href{mailto:ngn1@alumni.williams.edu}{\tt ngn1@alumni.williams.edu}   \\
\ \\
Sarah Peluse         \\
Department of Mathematics\\
Stanford University\\
Stanford, CA 94305\\
USA\\
\href{mailto:speluse@stanford.edu}{\tt speluse@stanford.edu}
\\
\ \\
Matthew Stoffregen\\
Department of Mathematics\\
University of California, Los Angeles\\
Los Angeles, CA 90095\\
USA   \\
\href{mailto:mstoffregen@math.ucla.edu}{\tt mstoffregen@math.ucla.edu }

\end{center}

\vskip .2in

\begin{abstract}
The Stern diatomic sequence is closely linked to continued fractions via the Gauss map on the unit interval, which in turn can be understood via systematic subdivisions of the unit interval.
 Higher dimensional analogues of continued fractions, called multidimensional continued fractions, can be produced through various subdivisions of a triangle.  We define triangle partition-Stern sequences (TRIP-Stern sequences for short) from certain triangle divisions developed earlier by the authors. These sequences are higher-dimensional generalizations of the Stern diatomic sequence. We then prove several combinatorial results about TRIP-Stern sequences, many of which give rise to well-known sequences. We finish by generalizing TRIP-Stern sequences and presenting analogous results for these generalizations.
\end{abstract}

\section{Introduction}

Stern's diatomic sequence (defined in Section \ref{CF.5}) stems from the study of continued fractions and has a  number of remarkable combinatorial properties, as seen in Northshield \cite{NorthshieldS10}. There are many different  multidimensional continued fraction algorithms, and they serve a number of different purposes ranging from   simultaneous Diophantine approximation problems (see Lagarias  \cite{Lagarias93}) to attempts to understand algebraic numbers via periodicity conditions (see the third author's  \cite{GarrityT01}) to automata theory (see Fogg \cite{Fogg}).  This paper concerns a generalization of Stern's diatomic sequence defined using triangle partition maps, a family of multidimensional continued fractions that includes most of the well-known multidimensional continued fractions presented in Schweiger \cite{Schweiger1}.  For background on multidimensional continued fractions, see Schweiger \cite{Schweiger1} and  Karpenkov \cite{Karpenkov}.

For background on the properties of Stern's diatomic sequence, see Lehmer \cite{Lehmer1}. 
For  background on Stern's diatomic sequence in the context of continued fractions, see Northshield \cite{NorthshieldS10}. Knauf found connections between Stern's diatomic sequence and statistical mechanics \cite{Knauf1, Knauf2, Knauf5, Knauf3, Knauf4}, though Knauf called the sequences \textit{Pascal with memory}, which is a more apt description.  The connection between Stern's sequence  and statistical mechanics was further developed in Contucci and Knauf \cite{Contucci-Knauf1}, Esposti, Isola and Knauf \cite{Esposti-Isola-Knauf1}, Fiala and Kleban  \cite{ Fiala-Kleban1}, Fiala, Kleban and Ozluk \cite{ Fiala-Kleban-Ozluk1},  Garrity \cite{ Garrity10}, Guerra and Knauf \cite{ Guerra-Knauf1},  Kallies, Ozluk, Peter and Syder\cite{ Kallies-Ozluk-Peter-Syder1},  Kleban and Ozluk \cite{ Kelban-Ozluk1}, Mayer \cite{ Mayer2}, Mend\`{e}s France and Tenenbaum \cite{MendesFrance-Tenenbaum1, MendesFrance-Tenenbaum2}, Prellberg, Fiala and Kleban  \cite{Prellberg-Fiala-Kleban1}, and Prellberg and Slawny  \cite{ Prellberg-Slawny1}.   Other important earlier work was done by Allouche and Shallit \cite{Allouche-Shallit92, Allouche-Shallit03}, who showed that  Stern's sequence is 2-regular.

There seems to have been little work on extending Stern's diatomic sequence to multidimensional continued fraction algorithms. The first generalization of Stern's diatomic sequence was for a  type of multidimensional continued fraction called the Farey map, in Garrity \cite{Garrity13}.  The Farey map is not one of the multidimensional continued fractions that we will be considering. Another generalization used the M\"{o}nkemeyer map, in Goldberg \cite{Goldberg12}.

This paper uses the family of multidimensional continued fractions called triangle partition maps (TRIP maps for short)  \cite{SMALL11q1} to construct analogous sequences. As mentioned, many, if not most, known multidimensional continued fraction algorithms can be put into the language of triangle partition maps; thus, the collection of TRIP maps is a rich family. In Section \ref{CF}, we give a quick overview of continued fractions, Stern's diatomic sequence and how the two are related.  Section \ref{S2}  reviews triangle partition maps and triangle partition sequences.   Section \ref{S2.5}  introduces the construction of TRIP-Stern sequences.  In Section \ref{pic}, we give a more pictorial description of TRIP-Stern sequences.  Section \ref{S3} contains results about the maximum terms and locations thereof for each level of the TRIP-Stern tree. Section \ref{S4} discusses minimum terms and locations thereof.  Section \ref{sums} examines sums of levels of the sequence.  Section \ref{S6} determines which lattice points appear in the TRIP-Stern sequence for the triangle map, a multidimensional fraction algorithm discussed below.  Section \ref{S7} introduces a generalization of the original TRIP-Stern sequence.
Finally, we close in Section \ref{Conclusion} with some of the many questions that remain.

\section{Continued fractions and Stern's diatomic sequence}\label{CF}
Nothing in this section is new.  In the first subsection, we review continued fractions in order to motivate, in part,  the definition of triangle partition maps given in Section \ref{S2}.  In the second subsection, we review
the classical Stern's diatomic sequence and show how it is linked to continued fractions. This link is what this paper generalizes.

\subsection{Continued fractions and subdivisions of the  unit interval}
All of the content in this subsection is well-known.  

Let $\alpha$ be a real number in the unit interval $I=(0,1]$.  The \textit{Gauss map} is the function 
$G: (0,1] \rightarrow [0,1)$ defined by
\[
G(\alpha) = \frac{1}{\alpha} - \left\lfloor \frac{1}{\alpha} \right\rfloor,
\]
where $\lfloor x \rfloor$ denotes the floor function, meaning the greatest integer less than or equal to $x$.
Subdivide the unit interval into subintervals 
\[
I_k = \left( \frac{1}{k+1}, \frac{1}{k}  \right]
\]
for $k$ a positive integer.  If $\alpha \in I_k$, then the Gauss map is simply $G(\alpha) = \frac{1 - k \alpha}{\alpha}$.
The continued fraction expansion  of $\alpha$ is
\[
\alpha = \frac{1}{  a_0 + \frac{1}{ a_1 + \frac{1}{    a_2 + \frac{1}{      \ddots                              }}}  }
\]
where $\alpha\in I_{a_0}, G(\alpha)\in I_{a_1}, G(G(\alpha)) \in I_{a_2}, \ldots$.   (If $\alpha$, under the iterations of $G$, is ever zero, then the algorithm stops.)  

We now want to translate the definition of the Gauss map into the language of two-by-two matrices, which can be more easily generalized.  Set 
\[  v_1 = \left( \begin{array}{c} 0 \\ 1 \end{array}  \right)  \text{ and }  v_2 = \left( \begin{array}{c} 1 \\ 1 \end{array}  \right).\]

  We have the standard identification of a vector in $\mathbb{R}^2$ to a real number via
  \[ \left( \begin{array}{c} x \\ y \end{array}  \right)  \rightarrow \frac{x}{y},\]
  provided of course that $y\neq 0$.
  Then we think of the two-by-two matrix 
  \[V = ( v_1, v_2) =  \left( \begin{array}{cc} 0 & 1 \\ 1& 1 \end{array}  \right) \]
  as being identified to the unit interval $I$.  
  Set 
  \[F_0 =  \left( \begin{array}{cc} 0 & 1 \\ 1& 1 \end{array}  \right)\text{ and } F_1 =  \left( \begin{array}{cc} 1 & 1 \\ 0& 1 \end{array}  \right).\]
  Then, by a calculation,  we have 
  \[VF_1^{k-1}F_0 =  \left( \begin{array}{cc} 1 & 1 \\k&k+1 \end{array}  \right),\]
  which can be identified to the subinterval $I_k $.  Further, by a calculation, we have that 
  \begin{eqnarray*}
  V(VF_1^{k-1}F_0)^{-1}  \left( \begin{array}{c} \alpha \\ 1 \end{array}  \right)  &=&  \left( \begin{array}{cc} -k & 1 \\1&0 \end{array}  \right) \left( \begin{array}{c} \alpha \\ 1 \end{array}  \right) \\
  &=&  \left( \begin{array}{c} -k\alpha + 1 \\ \alpha \end{array}  \right) \\
  &\rightarrow &  \frac{1 - k \alpha}{\alpha},
  \end{eqnarray*}
  and thus captures the Gauss map.  
  
  Using the matrices $F_0$ and $F_1$, we can also interpret the Gauss map as a method of systematically subdividing the unit interval.  This interpretation leads to the classical Stern diatomic sequence.   Note that
  \begin{eqnarray*}
  VF_0 &=& (v_1, v_2 ) F_0 \\
  &=& (v_1, v_2 )    \left( \begin{array}{cc} 0 & 1 \\ 1& 1 \end{array}  \right).  \\
  &=& (v_2, v_1 + v_2)  \\
  &=&  \left( \begin{array}{cc} 1 & 1 \\ 1& 2 \end{array}  \right)  \\
  \end{eqnarray*}
and
  \begin{eqnarray*}
  VF_1  &=& (v_1, v_2 ) F_1 \\
  &=& (v_1, v_2 )    \left( \begin{array}{cc} 1 & 1 \\ 0& 1 \end{array}  \right).  \\
  &=& (v_1, v_1 + v_2)  \\
  &=&  \left( \begin{array}{cc} 0 & 1 \\ 1& 2 \end{array}  \right)  
    \end{eqnarray*}
We can interpret $VF_0$ as the half interval $(1/2, 1)$ and $VF_1$ as the half interval $(0, 1/2)$. If we iterate multiplying by $F_0$ and $F_1$, then we get the following at the next step:
\[VF_1F_1 =   \left( \begin{array}{cc} 0 & 1 \\ 1& 3 \end{array}  \right),\;  VF_1F_0 =   \left( \begin{array}{cc} 1 & 1 \\ 2& 3 \end{array}  \right),\; VF_0F_1 =   \left( \begin{array}{cc} 1 & 2 \\ 1& 3 \end{array}  \right)\text{ and } VF_0F_0 =   \left( \begin{array}{cc} 1 & 2 \\ 2& 3 \end{array}  \right)  \]
Each real number $\alpha \in I$ can be described by a sequence $(i_0, i_1, i_2, \ldots )$ of zeros and ones, where, for all $n$, the number $\alpha$ lies in the subinterval coming from $VF_{i_0} F_{i_1} F_{i_2} \cdots F_{i_n}$.   (We are being somewhat sloppy with issues of $\alpha$ being on the boundaries of these subintervals.  Such issues do not affect what is going on.) We can link the sequence $(i_0, i_1, i_2, \ldots )$  with $\alpha$'s continued fraction expansion as follows. Let $1^k$ denote a sequence of $k$ ones.  Then our sequence $(i_0, i_1, i_2, \ldots )$ can be written as 
\[(i_0, i_1, i_2, \ldots ) = (1^{k_{0}}, 0, 1^{k_1},0,1^{k_2},0, \dots ),\]
with each $k_j$ a non-negative integer.  (It is important that we allow a $k_j$ to be zero.)  Then we have 
\[\alpha =  \frac{1}{  k_0  + 1+ \frac{1}{ k_1 +1+ \frac{1}{    k_2+1 + \frac{1}{      \ddots         }}}  }.\]

For example, the sequence $ (1,1,0,0,1,1,1,0, \ldots )$ can be written as $(1^2,0,1^0,0,1^3,0, \ldots)$, and we have 
\[\alpha =  \frac{1}{  3  + \frac{1}{  1+ \frac{1}{    4 + \frac{1}{      \ddots                              }}}  }.\]

Thus, continued fractions can be interpreted as a systematic method for subdividing an interval using two-by-two matrices.  Multi-dimensional continued fractions, as we will see, are systematic subdivisions of triangles determined by three-by-three matrices.

\subsection{Stern's diatomic sequence} \label{CF.5}
In this section, we will briefly review Stern's diatomic sequence (number \seqnum{A002487} in Sloane's \textit{Online Encyclopedia of Integer Sequences}). 
In particular, we highlight the link between Stern's diatomic sequence and continued fractions.   The classical \textit{Stern's diatomic sequence} $a_1,a_2, a_3, \ldots$ is the sequence defined by $a_1=  1$ and, for $n \geq 1$,
\begin{eqnarray*}
a_{2n} & = & a_n   \\
a_{2n + 1} &=&  a_n + a_{n+1}.
\end{eqnarray*}

Stern's diatomic sequence is linked to the Stern-Brocot array, which is an array  of fractions in
lowest terms that contains all rationals in the interval $[0,1]$.
Starting with the fractions $\frac{0}{1}$ and $\frac{1}{1}$ on the
$0^{\rm th}$ level, we construct the $n^{\rm th}$ level by rewriting
the $(n-1)^{\rm st}$ level with the addition of the mediant
between consecutive pairs of fractions from the $(n-1)^{\rm st}$ level.  Here, the \textit{mediant} of two fractions $\frac{a}{b}$
and $\frac{c}{d}$ refers to the fraction $\frac{a+c}{b+d}$. In the Stern-Brocot array, the mediant of consecutive fractions
is always in lowest terms. Below are levels $0$ through $3$ of the Stern-Brocot array:

\vspace{2mm}
\begin{center}
$\begin{array}{ccccccccc}
\vspace{2mm}
\frac{0}{1},&&&&&&&&\frac{1}{1}\\
\vspace{2mm}
\frac{0}{1},&&&&\frac{1}{2},&&&&\frac{1}{1}\\
\vspace{2mm}
\frac{0}{1},&&\frac{1}{3},&&\frac{1}{2},&&\frac{2}{3},&&\frac{1}{1}\\
\vspace{2mm}
\frac{0}{1},&\frac{1}{4},&\frac{1}{3},&\frac{2}{5},&\frac{1}{2},&\frac{3}{5},&\frac{2}{3},&\frac{3}{4},&\frac{1}{1}\\
\end{array}$
\end{center}

The denominators of the Stern-Brocot array form Stern's diatomic sequence.
Many of the combinatorial properties of this sequence are presented in Northshield \cite{NorthshieldS10}.

The first row of the array can be thought of as either the unit interval, or, as above, the two-by-two matrix $V = \left( \begin{array}{cc} 0 & 1 \\ 1& 1 \end{array}  \right)$.  The second row can be thought of as the two subintervals, $(0, 1/2)$ and $ (1/2, 1)$, or, as the two matrices  $VF_0$ and $VF_1$.  Similarly, the third row gives us four subintervals, each corresponding to one of the matrices
$VF_0F_0, VF_0F_1, VF_1F_0$ and $VF_1F_1$.  The pattern continues.  

To be more precise, let $s_{n,k}$ denote the $k^{\rm th}$ fraction in the $n^{\rm th}$
level of the Stern-Brocot tree. One can use the Stern-Brocot array to express the continued fraction expansion of a real number in $[0,1]$ as follows: 
let $\alpha\in[0,1]$. The $1^{\rm st}$ level of the Stern-Brocot tree divides the unit interval in two as the subintervals $[0,\frac{1}{2})$ and $[\frac{1}{2},0]$. Label the first interval 0 and the second interval 1. The $2^{\rm nd}$
level divides the unit interval into four subintervals: $[0,\frac{1}{3}),$
$[\frac{1}{3},\frac{1}{2}),$ $[\frac{1}{2},\frac{2}{3}),$ and $[\frac{2}{3},\frac{1}{1}]$.
 Label these intervals 00, 01, 10, and 11 respectively. The $n^{\rm th}$ level divides
the unit interval into $[0,s_{n,1}),\ldots,[s_{n,2^{n}},1]$.
We label the interval $[s_{n,k},s_{n,k+1})$ with a sequence of $0$'s and
$1$'s, where the first $n-1$ digits mark the label of the interval containing
$[s_{n,k},s_{n,k+1})$ on the $(n-1)^{\rm st}$ level, and
where the last digit is 0 or 1 depending on whether $[s_{n,k},s_{n,k+1})$
is in the left or right half of that interval, respectively. Recording
the infinite sequence of $0$'s and $1$'s that corresponds to any number $\alpha$ in $[0,1]$
yields a sequence encoding the continued fraction
expansion of $\alpha$, as in described in Northshield \cite{NorthshieldS10}.  Thus, Stern's sequence is linked to subdivisions of the unit interval. Our generalizations of  Stern's sequence will be linked to subdivisions of a triangle.

\section{Review of triangle partition maps}
\label{S2}

TRIP-Stern sequences can be interpreted geometrically in terms of subdivisions of a triangle. (This section closely follows Sections 2 and 3 from Dasaratha et al.\ \cite{SMALL11q1}.)  In this section, we describe the triangle division and triangle function, as defined in Garrity \cite{GarrityT01} and further developed in  Chen et al.\ \cite{GarrityT05} and Messaoudi et al.\  \cite{SchweigerF08}.  We then discuss how ``permutations'' of this triangle division generate a family of multidimensional continued fractions called triangle partition maps (TRIP maps for short) -- which were introduced in Dasaratha et al.\ \cite{SMALL11q1, SMALL11q3} -- and studied in Jensen \cite{Jensen} and in Amburg \cite{Amburg}. This will give us the needed notation to define TRIP-Stern sequences in the next section.

\subsection{The triangle division}\label{triangle}

The triangle division generalizes the method for computing continued fractions via the above systematic subdivision of the  unit interval from the previous section.  Instead of dividing the unit interval, we now use a 2-simplex, i.e., a triangle.  As discussed in earlier papers, this triangle division is just one of many generalizations of the continued fraction algorithm. Define
\[\triangle^* = \{(b_0,b_1,b_2): b_0 \geq b_1 \geq b_2 > 0 \}.\]
The set $\triangle^*$ can be thought of as a ``triangle" in $\mathbb{R}^3$, using the projection map $\pi: \mathbb{R}^3 \to \mathbb{R}^2$ defined by \[\pi(b_0,b_1,b_2) = \left(\frac{b_1}{b_0}, \frac{b_2}{b_0}\right).\] The image of $\triangle^*$ under $\pi$, \[\triangle = \{ (1,x,y) : 1 \geq x \geq y > 0 \},\]
is a triangle in $\mathbb{R}^2$ with vertices $(0,0),$ $(1,0)$, and $(1,1)$. Thus $\pi$ maps the vectors
\[ v_1 = \left(\begin{array}{c}1 \\0 \\0\end{array}\right), ~
v_2 = \left(\begin{array}{c}1 \\1 \\0\end{array}\right),~
v_3 = \left(\begin{array}{c}1 \\1 \\1\end{array}\right)\]
to the vertices of $\triangle$. The change of basis matrix from triangle coordinates to the standard basis is \[(v_1\ v_2\   v_3)= \begin{pmatrix}
1 & 1 & 1 \\ 
0 & 1 & 1 \\
0 & 0 & 1 \\
\end{pmatrix}.\]

Now consider the matrices
\[A_0 = \left(\begin{array}{ccc}0 & 0 & 1 \\1 & 0 & 0 \\0 & 1 & 1\end{array}\right)\text{ and }
A_1 = \left(\begin{array}{ccc}1 & 0 & 1 \\0 & 1 & 0 \\0 & 0 & 1\end{array}\right).\]

Applying $A_0$ and $A_1$ to $(v_1 \ v_2 \ v_3)$ yields
\[(v_1 \ v_2 \ v_3)A_0=(v_2\ v_3 \ v_1+v_3)\mbox{ and }(v_1\   v_2\ v_3)A_1=(v_1\  v_2\  v_1+v_3).\]

This gives a disjoint bipartition of $\triangle$ under the map $\pi$, as seen in the diagram below.
\begin{center}
\setlength{\unitlength}{.1 cm}
\begin{picture}(70,70)
\put(5,5){\line(1,0){60}}
\put(65,5){\line(0,1){60}}
\put(5,5){\line(1,1){60}}
\put(0,0){(0,0)}
\put(60,0){(1,0)}
\put(60,68){(1,1)}

\put(65,5){\line(-1,1){30}}

\end{picture}
\end{center}

This is the first step of the triangle division algorithm. Now consider the result of applying $A_1$ to the original vertices of the triangle $k$ times followed by applying $A_0$ once. The vertices of the original triangle $\triangle$ are thus mapped as follows:
\[\triangle_k = \{ (1,x,y) \in \triangle : 1 - x - ky \geq 0 > 1- x - (k+1)y\}.\]

 We let $T: \triangle_k \to \triangle$ denote a collection of maps, where each is the following bijection between the subtriangle $\triangle_k$ and $\triangle$: For any point $(x, y) \in \triangle_k$ under the standard basis, first change the basis to triangle coordinates by multiplying the coordinates by $(v_1 \ v_2 \ v_3)$, apply the inverse of $A_0$ and the inverse of $A_1^k$, and then finally change the basis back to the standard basis. That is,
\begin{equation}
\label{eq_tfunction}
T(x,y) = \pi( (1 \ x \ y)\left( (v_1 \ v_2 \ v_3)(A_0)^{-1} (A_1)^{-k} (v_1 \ v_2 \ v_3)^{-1}\right)^T). 
\end{equation}
Rewriting yields the following:
\[T(x,y):=\left(\frac{y}{x},\frac{1-x-ky}{x}\right)\]
where $k=\lfloor\frac{1-x}{y}\rfloor$.

This triangle map is analogous to the Gauss map for continued fractions. Using $T$, define the \emph{triangle sequence}, $(t_k)_{k\geq 0}$, for an element  $(\alpha, \beta) \in \triangle$ by setting $t_n$ to be the non-negative integer satisfying  $T^{(n)}(\alpha, \beta) \in \triangle_{t_n}$.  Thus,
\[(\alpha, \beta) \in \triangle_{t_0}, T(\alpha, \beta) \in \triangle_{t_1}, T(T(\alpha, \beta)) \in \triangle_{t_2}, \ldots \]

\subsection{Incorporating permutations}

The previously described triangle division partitions the triangle with vertices $v_1, v_2,$  and $v_3$ into triangles with vertices $v_2, v_3,$ and $v_1 + v_3$ and $v_1,v_2,$ and $v_1+v_3$. This process assigns a particular ordering of vertices to the vertices of the original triangle and the vertices of the two triangles produced.  By considering all possible permutations we generate a family of 216 maps, each corresponding to a partition of $\triangle$.

Specifically, we allow a permutation of the vertices of the initial triangle as well as a permutation of the vertices of the triangles obtained after applying $A_0$ and $A_1$. First, we permute the vertices by $\sigma \in S_3$ before applying either $A_0$ or $A_1$. Once we apply either $A_0$ or $A_1$, we then permute by either $\tau_0 \in S_3$ or $\tau_1 \in S_3$, respectively. This leads to the following definition:

\begin{definition}For every $(\sigma, \tau_0, \tau_1) \in S^3_3$, define
\[F_0 = \sigma A_0 \tau_0 \text{ and } F_1 = \sigma A_1 \tau_1\]
by thinking of $\sigma,$ $\tau_0$, and $\tau_1$ as column permutation matrices.
\end{definition}

In particular, we define the permutation matrices as follows:
$e=\left(
\begin{array}{ccc}
 1 & 0 & 0 \\
 0 & 1 & 0 \\
 0 & 0 & 1 \\
\end{array}
\right),$
$(12)=\left(
\begin{array}{ccc}
 0 & 1 & 0 \\
 1 & 0 & 0 \\
 0 & 0 & 1 \\
\end{array}
\right),$
$(13)=\left(
\begin{array}{ccc}
 0 & 0 & 1 \\
 0 & 1 & 0 \\
 1 & 0 & 0 \\
\end{array}
\right),$
$(23)=\left(
\begin{array}{ccc}
 1 & 0 & 0 \\
 0 & 0 & 1 \\
 0 & 1 & 0 \\  
\end{array}
\right),$
$(123)=\left(
\begin{array}{ccc}
 0 & 1 & 0 \\
 0 & 0 & 1 \\
 1 & 0 & 0 \\
\end{array}
\right),$ and
$(132)=\left(
\begin{array}{ccc}
 0 & 0 & 1 \\
 1 & 0 & 0 \\
 0 & 1 & 0 \\
\end{array}
\right)$.

Note that applying $F_0$ and $F_1$ partitions any triangle into two subtriangles. Thus, for any $(\sigma, \tau_0, \tau_1) \in S^3_3$, we can partition $\triangle$ using the matrices $F_0$ and $F_1$ instead of $A_0$ and $A_1$. This produces a map that is similar to, but not identical to, the triangle map from section \ref{triangle}. We call each of these maps a \textit{triangle partition map}, or \textit{TRIP map} for short. Because $\left|S_3\right|^3 = 216$, the family of triangle partition maps has 216 elements.

One of the main goals of  Dasaratha et al.\ \cite{SMALL11q1} is showing that this class  of triangle partition maps includes well-studied algorithms such as the M\"{o}nkemeyer map, and in combination, these triangle partition maps can be used to produce many other known algorithms, such as the Brun, Parry-Daniels,  G\"{u}ting, and fully subtractive algorithms. 

\subsection{TRIP sequences}\label{Tsequence}

We use the triangle division for a particular set of permutations to produce the corresponding triangle sequence. Recall the ``subtriangle $\triangle_k$" from the original triangle map. We generalize the definition of $\triangle_k$ as follows: for any $(\sigma, \tau_0, \tau_1)$, let $\triangle_k$ be the image of the triangle $\triangle$ under $F_1^kF_0$. We now define functions $T: \triangle \rightarrow \triangle$ mapping each subtriangle $\triangle_k$ bijectively to $\triangle$.

First, let us formalize our subtriangles.

\begin{definition}
Let $F_0$ and $F_1$ be generated from some triplet of permutations. We define $\bigtriangleup_n$ to be the triangle with vertices given by the columns of $ (v_1\ \ v_2\ \ v_3){F_1}^n{F_0}$.
\end{definition}
We are now ready to define triangle partition maps.
\begin{definition}
We define the \textit{triangle partition map} $T_{\sigma, \tau_0,\tau_1}$ by
\[T_{\sigma, \tau_0, \tau_1}(x, y) =\pi((1, x, y) ( (v_1\ \ v_2\ \ v_3) F_0^{-1} F_1^{-k}  (v_1\ \ v_2\ \ v_3)^{-1})^T) \]
when $ (x,y)\in\triangle_k$.
\end{definition}

We take the transpose of the matrix $ (v_1\ \ v_2\ \ v_3) F_0^{-1} F_1^{-k}  (v_1\ \ v_2\ \ v_3)^{-1}$ because our matrices have vertices as columns but they are multiplied by the row vector $(1,x,y)$. 

 This definition facilitates the following:

\begin{definition}
For any an $(\alpha, \beta)\in\triangle$, define $t_n$ to be the non-negative integer such that $\big[T_{\sigma, \tau_0, \tau_1}\big]^n(\alpha, \beta)$ is in $\triangle_{t_n}$. The \textit{TRIP sequence} of $(\alpha, \beta)$ with respect to $(\sigma, \tau_0, \tau_1)$ is $(t_k)_{k\geq 0}$.
\end{definition}

\section{TRIP-Stern sequences}
\label{S2.5}

\subsection{Construction of TRIP-Stern sequences}
\label{TRIPSternDef}
\begin{definition}
For any permutations $(\sigma,\tau_0,\tau_1)$ in $S_3$, the \textit{triangle partition-Stern sequence} (\textit{TRIP-Stern sequence} for short) of $(\sigma, \tau_0, \tau_1)$ is the unique sequence such that $a_1 = (1,1,1)$ and, for $n\geq 1$,

\[
\begin{cases}
a_{2n} = a_n \cdot F_0; \\
a_{2n+1} =a_n \cdot F_1.
\end{cases}
\]

The $n^{\rm th}$ level of the TRIP-Stern sequence is the set of $a_m$ with $2^{n-1} \leq m < 2^n$.
\end{definition}
Each choice of $(\sigma, \tau_0, \tau_1)$ produces some TRIP-Stern sequence. The $n^{\rm th}$ level terms of the TRIP-Stern sequence give the first coordinate of the vertices of the subtriangles of $\bigtriangleup$ after $n$ divisions. These terms are the denominators of the convergents of the triangle partition map defined by $(\sigma, \tau_0, \tau_1)$. Thus, TRIP-Stern sequences can be used to test when a sequence of triangle subdivisions converges to a unique point. This involves a simple definition and restatement of Theorem 7.2 in Dasaratha et al.\ \cite{SMALL11q1}.

\begin{definition} For each $n$-tuple $v= (i_1, \ldots , i_n)$ of 0's and 1's,  define 
\[\triangle(v) = 
(1, 1, 1)F_{i_1}F_{i_2} \cdots F_{i_n}. \]
\end{definition}

For a TRIP-Stern term $a_n$, the length of the binary representation $v$ such that $\triangle(v) = a_n$ gives the level of $a_n$. If there are multiple binary representations of $a_n$, then choose the representation that gives the $n^{\rm th}$ term of the sequence.

\subsection{Examples}

We start by examining the  TRIP-Stern sequence for the triangle map, which is given by the identity permutations $(\sigma, \tau_0, \tau_1) = (e,e,e)$. The first few terms of the sequence $(a_n)_{n\geq 1}$ are
\[(1, 1, 1), (1, 1, 2), (1, 1, 2), (1, 2, 3), (1, 1, 3), (1, 2, 3), (1, 1, 3), \ldots\]
We arrange these as follows:

\begin{center}
\setlength{\unitlength}{.05 cm}
\begin{picture}(200,70)

\put(63,70){$a_1=(1,1,1)$}
\put(80,65){\vector(-1,-1){25}}
\put(95,65){\vector(1,-1){25}}

\put(30,30){$a_2=(1,1,2)$}

\put(110,30){$a_3=(1,1,2)$}

\put(40,25){\vector(-2,-1){44}}
\put(-30,-7){$a_4=(1,2,3)$}

\put(50,25){\vector(1,-2){12}}
\put(30,-7){$a_5=(1,1,3)$}

\put(120,25){\vector(-1,-2){12}}
\put(88,-7){$a_6=(1,2,3)$}

\put(150,25){\vector(1,-2){12}}
\put(150,-7){$a_7=(1,1,3)$}

\put(53,53){$A_0$}

\put(110,53){$A_1$}

\put(3,15){$A_0$}

\put(57,15){$A_1$}

\put(103,15){$A_0$}

\put(158,15){$A_1$}

\end{picture}
\end{center}

or as

\begin{center}
\setlength{\unitlength}{.05 cm}
\begin{picture}(200,70)

\put(85,70){$\triangle$}
\put(80,65){\vector(-1,-1){25}}
\put(95,65){\vector(1,-1){25}}

\put(38,30){$\triangle(0)$}

\put(118,30){$\triangle(1)$}

\put(40,25){\vector(-2,-1){44}}
\put(-20,-7){$\triangle(00)$}

\put(50,25){\vector(1,-2){12}}
\put(48,-7){$\triangle(01)$}

\put(120,25){\vector(-1,-2){12}}
\put(95,-7){$\triangle(10)$}

\put(130,25){\vector(1,-2){12}}
\put(130,-7){$\triangle(11)$}

\put(53,53){$A_0$}

\put(110,53){$A_1$}

\put(3,15){$A_0$}

\put(57,15){$A_1$}

\put(103,15){$A_0$}

\put(138,15){$A_1$}

\end{picture}
\end{center}

Note that because the two triples on level $1$ are the same, the left and right subtrees are symmetric. To see the connection to the triangle division, recall the matrix
\[(v_1 \ v_2 \ v_3)= \begin{pmatrix}
1 & 1 & 1 \\ 
0 & 1 & 1 \\
0 & 0 & 1 \\
\end{pmatrix},\]
whose columns are the three vertices of $\bigtriangleup$.  Applying $A_0$ and $A_1$ to $(v_1 \ v_2 \ v_3)$ yields

   \[(v_1 \ v_2 \ v_3)A_0 = \begin{pmatrix} 1 &1 & 2 \\ 1 & 1 & 1 \\ 0 & 1 & 1 \end{pmatrix} \text{ and }(v_1\   v_2\ v_3)A_1 = \begin{pmatrix} 1 &1 & 2 \\ 0 & 1 & 1 \\ 0 & 0 & 1 \end{pmatrix}.\]
   
Repeating this process yields matrices at each step $n$ whose columns give the vertices of the subtriangles at the $n^{\rm th}$ division of $\bigtriangleup$.  The figure below shows this subdivision. Additionally, the top row of each of these $n^{\rm th}$ step matrices gives a term in the $n^{\rm th}$-level TRIP-Stern sequence for $(e,e,e)$.

\begin{center}
\includegraphics[width=2.5in]{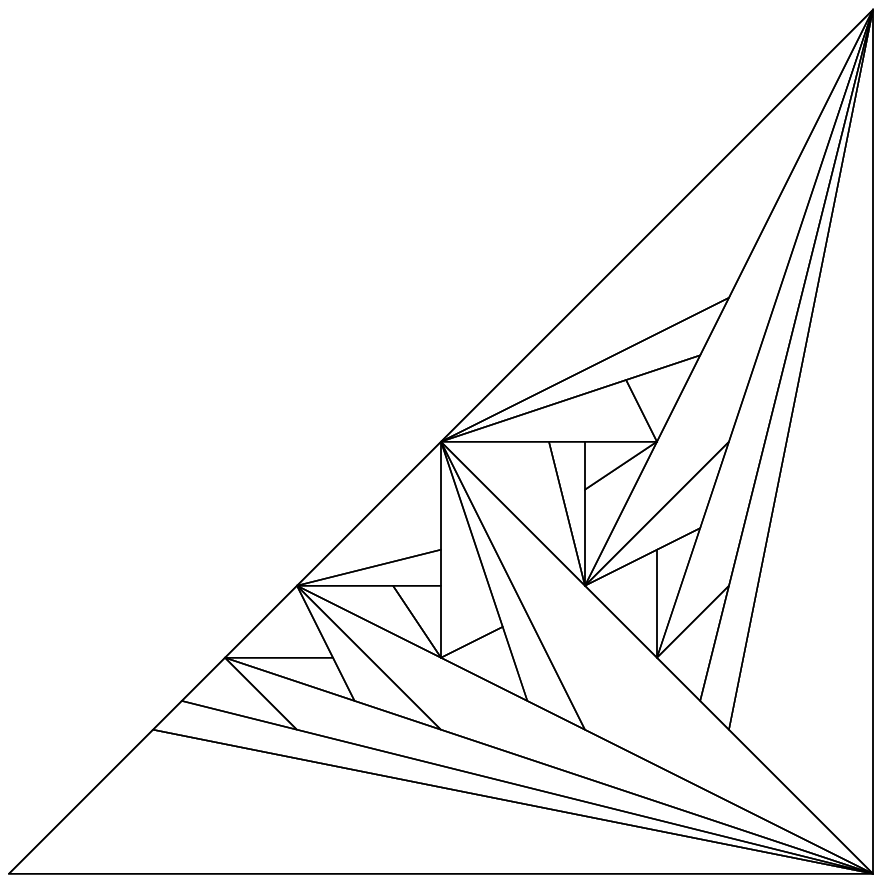}
\end{center}

For another example, consider the permutations $(\sigma, \tau_0, \tau_1) = \big(13,132,132\big)$. Recall here that we divide the triangle using the matrices $F_0=13A_0132$ and $F_1=13A_1132$.

Applying $F_0$ and $F_1$ to $(v_1 \ v_2 \ v_3)$ yields

   \[(v_1 \ v_2 \ v_3)F_0 = (v_1 \ v_2 \ v_3)\begin{pmatrix} 1 &1 & 0 \\ 0 & 0 & 1 \\ 0 & 1 & 0 \end{pmatrix}=\begin{pmatrix} 1 &2 & 1 \\ 0 & 1 & 1 \\ 0 & 1 & 0 \end{pmatrix}\] and \[(v_1\   v_2\ v_3)F_1 = (v_1 \ v_2 \ v_3)\begin{pmatrix} 0 &1 & 0 \\ 1 & 0 & 0 \\ 0 & 1 & 1 \end{pmatrix}=\begin{pmatrix} 1 &2 & 1 \\ 1 & 1 & 1 \\ 0 & 1 & 1 \end{pmatrix}.\]

These permutations give the well-studied M\"{o}nkemeyer map, as described in Dasaratha et al.\  \cite{SMALL11q1}.  Here, the first few terms are
\[(1, 1, 1), (1, 2, 1), (1, 2, 1), (1, 2, 2), (2, 2, 1), (1, 2, 2), (2, 2, 1), \ldots\]

Similarly, these terms now give the first coordinate of the vertices of the subtriangles of $\bigtriangleup$ according to the M\"{o}nkemeyer algorithm. The following figure shows this division.

\begin{center}
\includegraphics[width=2.5in]{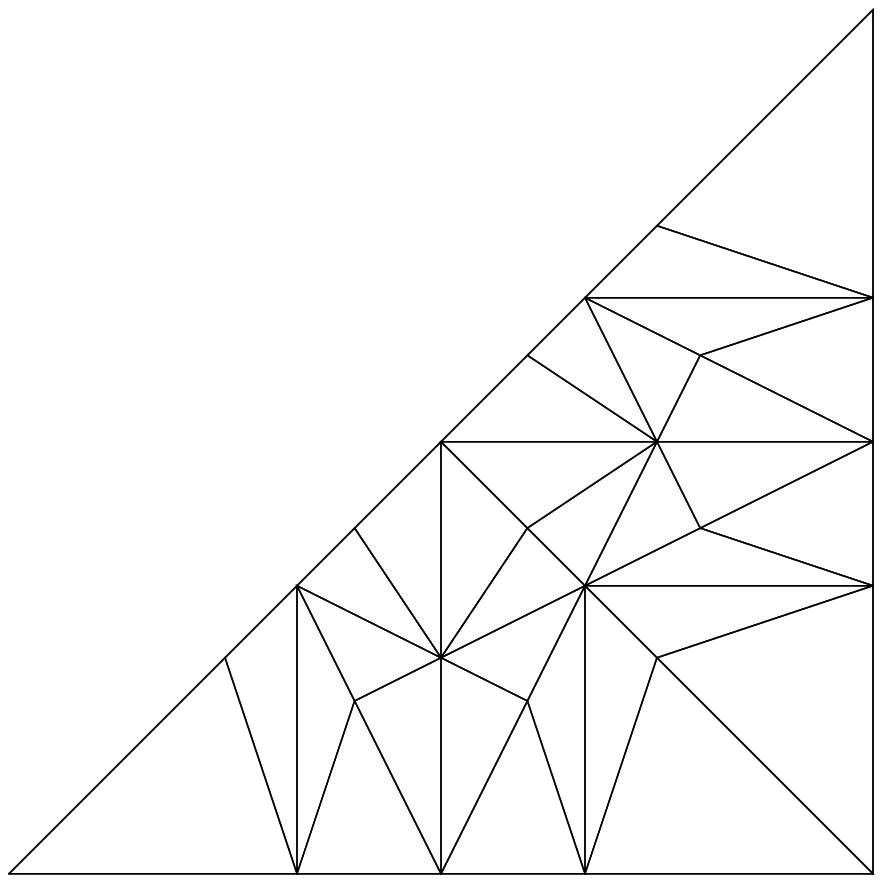}
\end{center}

We will examine TRIP-Stern sequences generated by each of the 216 maps. As we will see in Lemma \ref{onlye}, many properties of each of these TRIP-Stern sequences can be captured by only examining the 36 sequences associated with permutation triplets of the form $(e,\tau_0,\tau_1)$ for $\tau_0,\tau_1\in S_3$. The following table, which illustrates the behavior of $F_0$ and $F_1$ for maps of the form $(e,\tau_0,\tau_1)$ on the initial seed $(a,b,c)$ (which equals (1,1,1) in the case of regular TRIP-Stern sequences), will come in handy throughout the paper.

\begin{center}
$\begin{array}{l|c|c}
(e, \tau_0, \tau_1)  &  (a,b,c) F_0   &   (a,b,c) F_1 \\
\hline \hline
( e, e, e) & (b, c, a+c)  & (a, b, a+c) \\
( e, e, 12) & (b, c, a+c)  & (b, a, a+c) \\
( e, e, 13) & (b, c, a+c)  & (a+c, b, a) \\
( e, e, 23) & (b, c, a+c)  & (a, a+c, b) \\
( e, e, 123) & (b, c, a+c)  & (a+c, a, b) \\
( e, e, 132) & (b, c, a+c)  & (b, a+c, a) \\
( e, 12, e) & (c, b, a+c)  & (a, b, a+c) \\
( e, 12, 12) & (c, b, a+c) & (b, a, a+c) \\
( e, 12, 13) & (c, b, a+c) & (a+c, b, a) \\
( e, 12, 23) & (c, b, a+c)  & (a, a+c, b) \\
( e, 12, 123) & (c, b, a+c)  & (a+c, a, b) \\
( e, 12, 132) & (c, b, a+c)  & (b, a+c, a) \\

( e, 13, e) & (a+c, c, b)  & (a, b, a+c) \\
( e, 13, 12) & (a+c, c, b) & (b, a, a+c) \\
( e, 13, 13) & (a+c, c, b) & (a+c, b, a) \\
( e, 13, 23) &(a+c, c, b)  & (a, a+c, b) \\
( e, 13, 123) & (a+c, c, b)  & (a+c, a, b) \\
( e, 13, 132) & (a+c, c, b)  & (b, a+c, a) \\

( e, 23 , e) & (b, a+ c, c)  & (a, b, a+c) \\
( e,  23, 12) & (b, a+ c, c) & (b, a, a+c) \\
( e,  23, 13) & (b, a+ c, c) & (a+c, b, a) \\
( e,  23, 23) & (b, a+ c, c)  & (a, a+c, b) \\
( e,  23, 123) &(b, a+ c, c) & (a+c, a, b) \\
( e,  23, 132) &(b, a+ c, c)  & (b, a+c, a) \\

( e, 123, e) & (a+c, b, c)  & (a, b, a+c) \\
( e, 123, 12) &  (a+c, b, c)  & (b, a, a+c) \\
( e, 123, 13) &  (a+c, b, c) & (a+c, b, a) \\
( e, 123, 23) & (a+c, b, c)  & (a, a+c, b) \\
( e, 123, 123) & (a+c, b, c)  & (a+c, a, b) \\
( e, 123, 132) & (a+c, b, c) & (b, a+c, a) \\

( e, 132  , e) & (c,a+ c, b)  & (a, b, a+c) \\
( e, 132, 12) & (c,a+ c, b)   & (b, a, a+c) \\
( e, 132, 13) & (c,a+ c, b)   & (a+c, b, a) \\
( e, 132, 23) & (c,a+ c, b)  & (a, a+c, b) \\
( e, 132, 123) & (c,a+ c, b) & (a+c, a, b) \\
( e, 132, 132) & (c,a+ c, b) & (b, a+c, a) 
\end{array}$
\captionof{table}{Behavior of $F_0$ and $F_1$}
\end{center}

\subsection{Alternative definitions}

An alternative way to define the TRIP-Stern sequences is in terms of matrix generating functions. The advantage of this method is that it allows us to connect terms in a TRIP-Stern sequence with the product of $F_i$ matrices that produced the given sequence and to define the sequence non-recursively.

Let \[P(x) = F_0 + F_1 x,\] where $x$ commutes with $F_0$ and $F_1$. Now it is clear that any integer can be uniquely expressed as a sum of the form $2^n + k$, where $0 \leq k < 2^n$. 
\begin{definition}
 We define the \textit{TRIP-Stern sequence} for $(\sigma, \tau_0,\tau_1)$ to be the unique sequence defined by
\[a_{2^n + k} = (1, 1, 1) \cdot B,\]
where $B$ is the coefficient of $x^k$ in the product $P(x)P(x^2) \cdots P(x^n)$. Then $a_{2^n + k}$ is the $k^{\rm th}$ term of the $n^{\rm th}$ level of the TRIP-Stern sequence.
\end{definition}

For instance, take $(\sigma, \tau_0,\tau_1) = (e,e,e)$. Then \[P(x)P(x^2) = (A_0 + A_1x)(A_0 + A_1x^2) = A_0^2 + A_1A_0 x + A_0A_1x^2 + A_1^2 x^3\]
and so the terms of the $2^{\rm nd}$ level are given by \[(1,1,1)A_0^2, \ (1,1,1)A_1A_0,\ (1,1,1)A_0A_1, \text{ and }(1,1,1)A_1^2.\]

\section{A more pictorial approach}
\label{pic}
We need the technical definitions from the  last section  for the proofs given in the rest of the paper and in order to use Mathematica to discover many of our formulas.   But there is a more intuitive approach to generate particular TRIP-Stern sequences via subdivisions of a triangle. 

Let us first examine the classical Stern diatomic sequence analog.
Start with an interval $I$ whose endpoints are $a$ and $b$.  For Stern's diatomic sequence, we set $a=b=1$.  As before, we set $I=(a,b)$.

\begin{center}
$\setlength{\unitlength}{.1 cm}
\begin{picture}(70,20)
\put(35,15){$I$}
\put(10,10){\line(1,0){50}}
\put(10,9){\line(0,1){2}}
\put(60,9){\line(0,1){2}}
\put(10,5){$a$}
\put(60,5){$b$}
\end{picture}
\setlength{\unitlength}{.1 cm}
\begin{picture}(70,20)
\put(35,15){$I$}
\put(10,10){\line(1,0){50}}
\put(10,9){\line(0,1){2}}
\put(60,9){\line(0,1){2}}
\put(10,5){$1$}
\put(60,5){$1$}
\end{picture}$
\end{center}

We then subdivide the interval and add together the weights of the endpoints, getting two new intervals, which by an abuse of notation we write as  $I(0)=(a,a+b)$ and $I(1)=(a+b,b)$.

\begin{center}
$\setlength{\unitlength}{.1 cm}
\begin{picture}(70,20)
\put(20,15){$I(0)$}
\put(45,15){$I(1)$}
\put(10,10){\line(1,0){50}}
\put(10,9){\line(0,1){2}}
\put(35,9){\line(0,1){2}}
\put(60,9){\line(0,1){2}}
\put(10,5){$a$}
\put(30,5){$a+b$}
\put(60,5){$b$}
\end{picture}
\setlength{\unitlength}{.1 cm}
\begin{picture}(70,20)
\put(20,15){$I(0)$}
\put(45,15){$I(1)$}
\put(10,10){\line(1,0){50}}
\put(10,9){\line(0,1){2}}
\put(35,9){\line(0,1){2}}
\put(60,9){\line(0,1){2}}
\put(10,5){$1$}
\put(35,5){$2$}
\put(60,5){$1$}
\end{picture}$
\end{center}

We can continue, getting four new subintervals:

\begin{center}
$\setlength{\unitlength}{.1 cm}
\begin{picture}(70,20)
\put(12,16){$I(00)$}
\put(15,15){\vector(0,-1){4}}
\put(25,0){$I(01)$}
\put(29,4){\vector(0,1){4}}
\put(38,0){$I(10)$}
\put(41,4){\vector(0,1){4}}
\put(52,16){$I(11)$}
\put(55,15){\vector(0,-1){4}}
\put(10,10){\line(1,0){50}}
\put(10,9){\line(0,1){2}}
\put(26.7,9){\line(0,1){2}}
\put(35,9){\line(0,1){2}}
\put(43.3,9){\line(0,1){2}}
\put(60,9){\line(0,1){2}}
\put(10,5){$a$}
\put(21.7,12){$2a+b$}
\put(30,5){$a+b$}
\put(38.3,12){$a+2b$}
\put(60,5){$b$}
\end{picture}
\begin{picture}(70,20)
\setlength{\unitlength}{.1 cm}
\put(12,16){$I(00)$}
\put(15,15){\vector(0,-1){4}}
\put(25,0){$I(01)$}
\put(29,4){\vector(0,1){4}}
\put(38,0){$I(10)$}
\put(41,4){\vector(0,1){4}}
\put(52,16){$I(11)$}
\put(55,15){\vector(0,-1){4}}
\put(10,10){\line(1,0){50}}
\put(10,9){\line(0,1){2}}
\put(26.7,9){\line(0,1){2}}
\put(35,9){\line(0,1){2}}
\put(43.3,9){\line(0,1){2}}
\put(60,9){\line(0,1){2}}
\put(10,5){$1$}
\put(26.7,12){$3$}
\put(35,5){$2$}
\put(43.3,12){$3$}
\put(60,5){$1$}
\end{picture}$
\end{center}
The classical Stern diatomic sequence corresponds to  $I, I(0), I(1), I(00), I(01),I(10),I(11), \ldots$

We now see how to get the analogous geometric picture for TRIP maps.  We will concentrate on the TRIP-Stern sequence for the triple $(e,e,e)$.  Similar pictures, though, will work for any $(\sigma,\tau_0, \tau_1)\in S_{3}^{3}$.

For $(e,e,e)$ we have $F_0=A_0, F_1=A_1$. For any vector $(a,b,c)$, we know that 
\[(a,b,c)A_0 = (b,c,a+c)\text{ and } (a,b,c) A_1 = (a,b,a+c).\]
We will think of $(a,b,c)$ as vertices of a triangle $\triangle$.  By a slight abuse of notation, we will write $\triangle = (a,b,c)$. In the diagrams for this section, the ones on the left side are for the general case of any $a$, $b,$ and $c$, while the right side is in the special case when $a=b=c=1$.  We have

\[\begin{array}{lr}
\setlength{\unitlength}{.05 cm}
\begin{picture}(200,70)
\put(5,5){\line(1,0){60}}
\put(65,5){\line(0,1){60}}
\put(5,5){\line(1,1){60}}
\put(0,-.1){$a$}
\put(69,0){$b$}
\put(60,68){$c$}

\put(-3,30){$\triangle$}
\put(6,32){\vector(1,0){40}}

\end{picture}
& 
\setlength{\unitlength}{.05 cm}
\begin{picture}(70,70)
\put(5,5){\line(1,0){60}}
\put(65,5){\line(0,1){60}}
\put(5,5){\line(1,1){60}}
\put(0,-.1){1}
\put(69,0){1}
\put(60,68){1}

\put(-3,30){$\triangle$}
\put(6,32){\vector(1,0){40}}

\end{picture}
\end{array}\]

We let   $\triangle(0)$ be the triangle with vertices $(a,b,c)A_0 = (b,c,a+c)$ and  $\triangle(1)$ be the triangle with vertices $(a,b,c) A_1 = (a,b,a+c)$.  
Pictorially, we have 

\[\begin{array}{lr}
\setlength{\unitlength}{.05 cm}
\begin{picture}(200,70)
\put(5,5){\line(1,0){60}}
\put(65,5){\line(0,1){60}}
\put(5,5){\line(1,1){60}}
\put(0,-.1){$a$}
\put(69,0){$b$}
\put(60,68){$c$}

\put(16,35){$a+c$}
\put(65,5){\line(-1,1){30}}

\put(3,55){$\triangle(0)$}
\put(22,55){\vector(2,-1){28}}

\put(-12,18){$\triangle(1)$}
\put(6,20){\vector(1,0){22}}

\end{picture}
&
\setlength{\unitlength}{.05 cm}
\begin{picture}(70,70)
\put(5,5){\line(1,0){60}}
\put(65,5){\line(0,1){60}}
\put(5,5){\line(1,1){60}}
\put(0,-.1){1}
\put(69,0){1}
\put(60,68){1}

\put(28,35){2}
\put(65,5){\line(-1,1){30}}

\put(-12,18){$\triangle(1)$}
\put(6,20){\vector(1,0){22}}

\put(3,55){$\triangle(0)$}
\put(22,55){\vector(2,-1){28}}

\end{picture}
\end{array}\]

We let  $\triangle(00)$ be the triangle with vertices $(a,b,c)A_0A_0 =(c, a+c, a+b+c)$,  $\triangle(01)$ be the triangle with vertices $(a,b,c)A_0 A_1 = (b,c,a+b+c)$,   $\triangle(10)$ be the triangle with vertices $(a,b,c)A_1A_0 =(b, a+c, 2a+c)$ and $\triangle(11)$ be the triangle with vertices $(a,b,c)A_1A_1 =(a,b, 2a+c)$

Pictorially, we have

\[\begin{array}{lr}
\setlength{\unitlength}{.05 cm}
\begin{picture}(200,70)
\put(5,5){\line(1,0){60}}
\put(65,5){\line(0,1){60}}
\put(5,5){\line(1,1){60}}
\put(0,-.1){$a$}
\put(69,0){$b$}
\put(60,68){$c$}

\put(16,35){$a+c$}
\put(65,5){\line(-1,1){30}}

\put(3,55){$\triangle(00)$}
\put(25,55){\vector(2,-1){24}}

\put(-16,12){$\triangle(11)$}
\put(7,14){\vector(1,0){14}}

\put(65,5){\line(-2,1){40}}

\put(65,65){\line(-2, -5){17}}

\put(75,20){$a+b+c$}
\put(70,22){\vector(-1,0){21}}

\put(0,23){$2a+c$}

\put(72,37){$\triangle(01)$}
\put(70,40){\vector(-1,0){10}}

\put(30,-10){$\triangle(10)$}
\put(38,0){\vector(0,1){23}}

\end{picture}
&
\setlength{\unitlength}{.05 cm}
\begin{picture}(70,70)
\put(5,5){\line(1,0){60}}
\put(65,5){\line(0,1){60}}
\put(5,5){\line(1,1){60}}
\put(0,-.1){1}
\put(69,0){1}
\put(60,68){1}

\put(28,35){2}
\put(65,5){\line(-1,1){30}}

\put(3,55){$\triangle(00)$}
\put(25,55){\vector(2,-1){24}}

\put(-16,12){$\triangle(11)$}
\put(7,14){\vector(1,0){14}}

\put(65,5){\line(-2,1){40}}

\put(65,65){\line(-2, -5){17}}

\put(75,20){$3$}
\put(70,22){\vector(-1,0){21}}

\put(17,23){$3$}

\put(72,37){$\triangle(01)$}
\put(70,40){\vector(-1,0){10}}

\put(30,-10){$\triangle(10)$}
\put(38,0){\vector(0,1){23}}

\end{picture}
\end{array}\]

The TRIP-Stern sequence for $(e,e,e)$ is simply $\triangle$,  $  \triangle(0)$,  $ \triangle(1)$,  $  \triangle(00)$,  $  \triangle(01)$,  $  \triangle(10)$,  $  \triangle(11)$,  $ \ldots$

We just add the appropriate vertices to generate new subdivisions.  This is in direct analog to the classical Stern diatomic sequence.

As another example, let us look at the triangle partition Stern sequence for $(12,e,e)$.  We first need the matrices $F_0$ and $F_1$:

\[F_0 = (12)A_0  
=\left(
\begin{array}{ccc}
 0 & 1 & 0 \\
 1 & 0 & 0 \\
 0 & 0 & 1 \\
\end{array}
\right)  \left(
\begin{array}{ccc}
 0 & 0 & 1 \\
 1 & 0 & 0 \\
 0 & 1 & 1 \\
\end{array}
\right)  
= \left(
\begin{array}{ccc}
 1 & 0 & 0 \\
 0 & 0 & 1 \\
 0 & 1 & 1 \\
\end{array}
\right) \]
and
\[F_1 = (12)A_1 
= \left(
\begin{array}{ccc}
 0 & 1 & 0 \\
 1 & 0 & 0 \\
 0 & 0 & 1 \\
\end{array}
\right)  \left(
\begin{array}{ccc}
 1 & 0 & 1 \\
 0 & 1 & 0 \\
 0 & 0 & 1 \\
\end{array}
\right)  
= \left(
\begin{array}{ccc}
 0 & 1 & 0 \\
 1 & 0 & 1 \\
 0 & 0 & 1 \\
\end{array}
\right) \]

Now for any vector $(a,b,c)$,  we have  
\[(a,b,c)F_0 = (a,c,b+c)\text{ and } (a,b,c) F_1 = (b,a,b+c).\]
We still let $\triangle$ have vertices $a$, $b$, and $c$, but now the subtriangle $\triangle(0)$ will have vertices $a$, $c$, and $b+c$ and the  subtriangle $\triangle(1)$ will have vertices $a$, $b$, and $b+c$.
Pictorially, we have

\[\begin{array}{lr}
\setlength{\unitlength}{.05 cm}
\begin{picture}(200,70)
\put(5,5){\line(1,0){60}}
\put(65,5){\line(0,1){60}}
\put(5,5){\line(1,1){60}}
\put(0,-.1){$a$}
\put(69,0){$b$}
\put(60,68){$c$}

\put(68,35){$b+c$}
\put(5,5){\line(2,1){60}}

\put(3,55){$\triangle(0)$}
\put(22,55){\vector(2,-1){28}}

\put(-12,18){$\triangle(1)$}
\put(6,20){\vector(1,0){45}}

\end{picture}
&
\setlength{\unitlength}{.05 cm}
\begin{picture}(70,70)
\put(5,5){\line(1,0){60}}
\put(65,5){\line(0,1){60}}
\put(5,5){\line(1,1){60}}
\put(0,-.1){1}
\put(69,0){1}
\put(60,68){1}

\put(68,35){2}
\put(5,5){\line(2,1){60}}

\put(-12,18){$\triangle(1)$}
\put(6,20){\vector(1,0){45}}

\put(3,55){$\triangle(0)$}
\put(22,55){\vector(2,-1){28}}

\end{picture}
\end{array}\]

Let us do one more iteration.
We have that $\triangle(0,0)$ will have vertices   $(a, b+c, b + 2c)
   $, $ \triangle(0,1)$ will have vertices $  (c, a, b+2c)   $, $ \triangle(1,0)$ will have vertices $(b, b+c, a+b+c)    $ and $ \triangle(1,1)$ will have vertices $ (a, b, a+b+c)   $ since
\begin{eqnarray*}
(a,b,c) F_0 F_0 &=&  (a, c, b+c) F_0 \\
&=& (a, b+c, b + 2c) , \\
(a,b,c)F_0F_1 &= &(a,c, b + c) F_1 \\
&=& (c, a, b+2c),  \\
(a,b,c) F_1F_0 &=& ( b, a, b+c)F_0 \\
&=& (b, b+c, a+b+c),  \\
\end{eqnarray*}
and
\begin{eqnarray*}
(a,b,c ) F_1F_1 &=& (b,a,b+c)F_1 \\
&=& (a, b, a+b+c).
\end{eqnarray*}

By continuing these subdivisions, we get the TRIP-Stern sequence for $(12,e,e)$. Pictorially, we now  have

\[\begin{array}{lr}
\setlength{\unitlength}{.05 cm}
\begin{picture}(200,70)
\put(5,5){\line(1,0){60}}
\put(65,5){\line(0,1){60}}
\put(5,5){\line(1,1){60}}
\put(0,-.1){$a$}
\put(69,0){$b$}
\put(60,68){$c$}

\put(68,35){$b+c$}
\put(5,5){\line(2,1){60}}

\put(5,5){\line(4,3){60}}
\put(68,48){$b+2 c$}

\put(23,-10){$a + b+ c$}
\put(45,-5){\vector(0,1){28}}

\put(65,4){\line(-1,1){20}}

\put(3,55){$\triangle(01)$}
\put(22,55){\vector(2,-1){28}}

\put(-12,35){$\triangle(00)$}
\put(6,35){\vector(1,0){45}}

\put(-12,18){$\triangle(10)$}
\put(6,20){\vector(1,0){49}}

\put(-15,10){$\triangle(11)$}
\put(3,10){\vector(1,0){24}}

\end{picture}
&
\setlength{\unitlength}{.05 cm}
\begin{picture}(70,70)
\put(5,5){\line(1,0){60}}
\put(65,5){\line(0,1){60}}
\put(5,5){\line(1,1){60}}
\put(0,-.1){1}
\put(69,0){1}
\put(60,68){1}

\put(68,35){2}
\put(5,5){\line(2,1){60}}

\put(68,35){2}
\put(5,5){\line(2,1){60}}

\put(5,5){\line(4,3){60}}
\put(68,48){3}

\put(40,-12){ 3}
\put(45,-5){\vector(0,1){28}}

\put(3,55){$\triangle(01)$}
\put(22,55){\vector(2,-1){28}}

\put(-12,35){$\triangle(00)$}
\put(6,35){\vector(1,0){45}}

\put(-12,18){$\triangle(10)$}
\put(6,20){\vector(1,0){49}}

\put(-15,10){$\triangle(11)$}
\put(3,10){\vector(1,0){24}}

\put(65,4){\line(-1,1){20}}
\end{picture}
\end{array}\]

\section{Maximum terms and positions thereof}
\label{S3}

This section examines the maximum terms in every level of a TRIP-Stern sequence, as well as the positions of those maximum terms within the level. Recall for an $n$-tuple $v= (i_1, \ldots , i_n)$ of 0's and 1's that $\triangle(v) = (1, 1, 1)F_{i_1}F_{i_2} \cdots F_{i_n} $, which can be written as $\triangle(v)=(b_1(v),b_2(v),b_3(v))$. Let $|v|$ denote the number of entries in $v$.

\begin{definition}
 The \textit{maximum entry} on level $n$ of a TRIP-Stern sequence is \[m_n = \max_{|v| = n\:\:} \max_{i\in\{1,2,3\}}b_i(v).\]
 \end{definition}

Thus, for example, the sequence $(m_n)_{n\geq 0}$ for the TRIP-Stern sequence for $(e,e,e)$ begins
\[(1,2,3,4,6,9,13,19,28,\ldots).\] We can easily extend this kind of analysis to all 216 TRIP-Stern sequences.  Numerically, it appears that there are only eight possible row maxima sequences for all 216 TRIP-Stern sequences.   The following lemma will be needed before presenting our results on maximum row sequences. 
\begin{lemma}
\label{onlye}
Let $(\sigma,\tau_0,\tau_1)\in S_{3}^3$ have row maxima sequence $(m_n)_{n\geq 0}$ and suppose $\kappa\in S_3$. Then $(\kappa \sigma, \tau_0 \kappa^{-1}, \tau_1 \kappa^{-1})$ also has row maxima sequence $(m_n)_{n\geq 0}$.
\end{lemma}

\begin{proof}
This follows by direct calculation, since for any $v=(i_1, \ldots, i_n)$,
\begin{align*}
\triangle_{(\kappa \sigma, \tau_0 \kappa^{-1}, \tau_1 \kappa^{-1})} (v)& = (1,1,1)\cdot \kappa \sigma A_{i_1} \tau_{i_1} \kappa^{-1} \cdot \kappa \sigma A_{i_2} \tau_{i_2}  \kappa^{-1} \cdots \kappa \sigma A_{i_n} \tau_{i_n} \kappa^{-1} \\
&= (1,1,1)\cdot \sigma A_{i_1} \tau_{i_1} \cdot \sigma A_{i_2}\tau_{i_2} \cdots \sigma A_{i_n} \tau_{i_n} \kappa^{-1}\\
&=\triangle_{(\sigma,\tau_0,\tau_1)}(v)\kappa^{-1}.
\end{align*}
Since $\kappa^{-1}$ is just a permutation, it follows that the maximal component of $\triangle_{(\kappa \sigma, \tau_0 \kappa^{-1}, \tau_1 \kappa^{-1})} (v)$ is the same as that of $\triangle_{(\sigma,\tau_0,\tau_1)}(v)$.
\end{proof}

Recall from Section \ref{S2.5} that the $(n+1)^{\rm st}$ level of a TRIP-Stern sequence is generated from the $n^{\rm th}$ by applying $F_0$ or $F_1$ to the sequence at the $m^{\rm th}$ term at level $n$. Denote by $n_m$ the position of this term. In other words, a TRIP-Stern sequence is a binary tree with exactly 2 children at each node $n_m,$ with the first (left) child generated by applying $F_0$ to the triplet at $n_m$ and the second generated by applying $F_1$. We will show that there exist three distinct paths through the trees containing all the maximal terms. We will then present recurrence relations for select row maxima sequences.

First, we show that there exist three distinct paths through the trees generated by triangle partition maps of the form $(e,\tau_0,\tau_1)$ containing the maximal terms. For each of the three classes of path, we will write out an explicit proof for one map generating such path; proofs for the rest of the maps are similar simple calculations and will be omitted.

\begin{theorem}
\label{paths}
There exists a path through each of the trees generated by 26 triangle partition maps of the form $(e,\tau_0,\tau_1)$ that contains the maximal terms. Namely, these paths are as follows:
\begin{enumerate} 
	\item For the eleven TRIP-Stern sequences $(e,e,13),$$(e,13,13),$$(e,13,123),$  $(e,23,13),$
$(e,23,123),$
$(e,23,132),$
$(e,123,13),$
$(e,123,123),$
$(e,132,13),$
$(e,132,123)$ and
$(e,132,132),$ always select the right edge.
	\item For the twelve TRIP-Stern sequences  $(e,e,e),$ $(e,e,12),$
$(e,e,23),$
$(e,e,123),$
$(e,e,132),$
$(e,12,e),$
$(e,12,12),$
$(e,12,13),$
$(e,12,23),$
$(e,12,123),$
$(e,12,132)$ and
$(e,132,23),$  always select the left edge.
	\item For the three TRIP sequences $(e, 13, 12)$, $(e,23,23)$ and
$(e,123,e),$
alternate between first selecting the left edge and then the right edge.
\end{enumerate} 
\end{theorem}

\begin{proof}

\begin{enumerate} 
	\item Always select the right edge.
	
		Consider $(e,13,123)$. Then $(x_1,x_2,x_3)F_0=(x_1+x_3,x_3,x_2)$ and $(x_1,x_2,x_3)F_1=(x_1+x_3,x_1,x_2)$. Since at the zeroth step we start with $(x_1,x_2,x_3)=(1,1,1),$ it is clear that the rightmost triplet will contain the maximal term at each level $n$ since this will lead to the greatest rates of growth for each of $x_{1_n},$ $x_{2_n}$ and $x_{3_n}$ as $n$ increases. 
	
Similar arguments can be made for $(e,e,13),$ $(e,13,13),$
$(e,23,13),$
$(e,23,123),$
$(e,23,132),$
$(e,123,13),$
$(e,123,123),$
$(e,132,13),$
$(e,132,123),$ and
$(e,132,132)$.

	\item Always select the left edge.
	
	Consider $(e,e,e)$. Then $(x_1,x_2,x_3)F_0=(x_2,x_3,x_1+x_3)$ and $(x_1,x_2,x_3)F_1=(x_1,x_2,x_1+x_3)$. Since at the zeroth step we start with $(x_1,x_2,x_3)=(1,1,1),$ it is clear that the leftmost triplet will contain the maximal term at each level $n$ since this will lead to the greatest rates of growth for each of $x_{1_n},$ $x_{2_n}$ and $x_{3_n}$ as $n$ increases. 	

Similar arguments can be made for $(e,e,12),$
$(e,e,23),$
$(e,e,123),$
$(e,e,132),$
$(e,12,e),$
$(e,12,12),$
$(e,12,13),$
$(e,12,23),$
$(e,12,123),$
$(e,12,132),$ and
$(e,132,23)$.

	\item Alternate between first selecting the left edge and then the right edge.
	
		Consider $(e,13,12)$. Then $(x_1,x_2,x_3)F_0=(x_1+x_3,x_3,x_2)$ and $(x_1,x_2,x_3)F_1=(x_2,x_1,x_1+x_3)$. Since at the zeroth step we start with $(x_1,x_2,x_3)=(1,1,1),$ it is clear that the triplet containing the maximal term at each level will lie on the nodes of the path generated by alternating between first selecting the left edge and then the right edge. 
		
Similar arguments can be made for
$(e,23,23),$ and
$(e,123,e)$.

\end{enumerate} 

\end{proof}

In the above theorem, we have shown that the maximal TRIP-Stern sequence lies on one of three possible paths for 26 TRIP-Stern sequences generated by 26 triangle partition maps. Using Lemma \ref{onlye} brings this total up to $26\cdot 6=156$ maps.

\begin{question}What are the paths for finding maximal terms for the remaining  60 TRIP-Stern sequences?
\end{question}

\subsection{Explicit formulas and recurrence relations for sequences of maximal terms}
In the above subsection we addressed the positions of maximal terms. Here we present formulas and recurrence relations that may be used to find the actual values of the maximal terms for 120 TRIP-Stern sequences.  

\begin{theorem}

The $n^{\rm th}$ maximal term $m_n$ in the TRIP-Stern sequences corresponding to the permutation triplets $(e,e,13),$ $(e,12,e),$ $(e,12,12),$ $(e,12,13),$ $(e,12,23),$ $(e,12,123),$ $(e,12,132),$ $(e,13,13),$ $(e,23,13),$ $(e,123,13),$ and $(e,132,13)$ is given by the formula \[m_n=\frac{2^{-n-1} \left(\left(\sqrt{5}-3\right) \left(1-\sqrt{5}\right)^n+\left(3+\sqrt{5}\right)
   \left(1+\sqrt{5}\right)^n\right)}{\sqrt{5}},\] which corresponds to the Fibonacci recurrence relation $m_{n}=m_{n-1}+m_{n-2}$ (\seqnum{A000045}).

\end{theorem}

\begin{proof}

By Theorem \ref{paths}, we know that the third term in each of the triplets given by following the left-most path in the tree generated by select permutation triplets is a maximal term; similarly, the first term in each of the triplets given by following the right-most path in the tree generated by select permutation triplets is a maximal term. Hence, for sequences of maxima found by following the left-most path, all that remains to find $m_n$ is to find the third term in the triplet $(1,1,1)F_{0}^{n}$; for each of the permutation triplets  listed in the theorem that are generated by choosing the left-most path, this third term corresponds to the desired explicit formula. Similarly, for sequences of maxima found by following the right-most path, all that remains to find $m_n$ is to find the first term in the triplet $(1,1,1)F_{1}^{n}$; for each of the permutation triplets  listed in the theorem that are generated by choosing the right-most path, this first term corresponds to the desired explicit formula. It is easy to check using standard methods, as in Matthews \cite{Recurrence}, that $m_{n}=m_{n-1}+m_{n-2}$.

As an example, let us consider $(e,e,13),$ for which \[F_1=\left(
\begin{array}{ccc}
 1 & 0 & 1 \\
 0 & 1 & 0 \\
 1 & 0 & 0 \\
\end{array}
\right).\]

Then
\[
F_{1}^{n}=\left(
\begin{array}{ccc}
 \frac{2^{-n-1} \left(-\left(1-\sqrt{5}\right)^{n+1}+\left(1+\sqrt{5}\right)^{n+1}\right)}{\sqrt{5}} & 0 & \frac{2^{-n}
   \left(-\left(1-\sqrt{5}\right)^n+\left(1+\sqrt{5}\right)^n\right)}{\sqrt{5}} \\
   && \\
 0 & 1 & 0 \\
 &&\\
 \frac{2^{-n} \left(-\left(1-\sqrt{5}\right)^n+\left(1+\sqrt{5}\right)^n\right)}{\sqrt{5}} & 0 & \frac{2^{-n-1} \left(\left(1-\sqrt{5}\right)^n
   \left(1+\sqrt{5}\right)+\left(-1+\sqrt{5}\right) \left(1+\sqrt{5}\right)^n\right)}{\sqrt{5}} \\
\end{array}
\right),\]
so that 

$(1,1,1)F_{1}^n=
\left(\frac{2^{-n-1} \left(\left(\sqrt{5}-3\right) \left(1-\sqrt{5}\right)^n+\left(3+\sqrt{5}\right)
   \left(1+\sqrt{5}\right)^n\right)}{\sqrt{5}},1,\frac{2^{-n-1} \left(\left(1+\sqrt{5}\right)^{n+1}-\left(1-\sqrt{5}\right)^{n+1}\right)}{\sqrt{5}}\right)$.

Hence, we can see that \[m_n=\frac{2^{-n-1} \left(\left(\sqrt{5}-3\right) \left(1-\sqrt{5}\right)^n+\left(3+\sqrt{5}\right)
   \left(1+\sqrt{5}\right)^n\right)}{\sqrt{5}}.\]
	
Lastly, it is clear that $m_n=m_{n-1}+m_{n-2}$ since
\[
m_{n-1}+m_{n-2}=\]
$\frac{2^{-1-(n-1)} \left(\left(1-\sqrt{5}\right)^{n-1} \left(-3+\sqrt{5}\right)+\left(1+\sqrt{5}\right)^{n-1}
   \left(3+\sqrt{5}\right)\right)}{\sqrt{5}}+\frac{2^{-1-(n-2)} \left(\left(1-\sqrt{5}\right)^{n-2} \left(-3+\sqrt{5}\right)+\left(1+\sqrt{5}\right)^{n-2}
   \left(3+\sqrt{5}\right)\right)}{\sqrt{5}}=$
	
	\[\frac{2^{-n-1} \left(\left(\sqrt{5}-3\right) \left(1-\sqrt{5}\right)^n+\left(3+\sqrt{5}\right)
   \left(1+\sqrt{5}\right)^n\right)}{\sqrt{5}}=m_n.\]

\end{proof}

\begin{theorem}

The $n^{\rm th}$ maximal term $m_n$ in the TRIP-Stern sequences corresponding to the permutation triplets $(e,e,e)$, $(e,e,12)$, $(e,e,23)$, $(e,e,123)$, $(e,e,132)$, $(e,13,123)$, $(e,23,123)$, $(e,123,123),$ and $(e,132,123)$  is given by the formula \[m_n=\alpha_1\beta_{1}^n+\alpha_2\beta_{3}^n+\alpha_3\beta_{2}^n,\] where the $\alpha_i$'s are roots of $31x^3-31x^2-12x-1=0$ while the $\beta_i$'s are the roots of $x^3-x^2-1=0$. This corresponds to the well-known recurrence relation $m_{n}=m_{n-1}+m_{n-3}$ (\seqnum{A000930}).

\end{theorem}

\begin{proof}

The proof, except for some technical details, is identical to the previous proof. By Theorem \ref{paths}, we know that the third term in each of the triplets given by following the left-most path in the tree generated by select permutation triplets is a maximal term; similarly, the first term in each of the triplets given by following the right-most path in the tree generated by select permutation triplets is a maximal term. Hence, for sequences of maxima found by following the left-most path, all that remains to find $m_n$ is to find the third term in the triplet $(1,1,1)F_{0}^{n}$; for each of the permutation triplets  listed in the theorem that are generated by choosing the left-most path, this third term corresponds to the desired explicit formula. Similarly, for sequences of maxima found by following the right-most path, all that remains to find $m_n$ is to find the first term in the triplet $(1,1,1)F_{1}^{n}$; for each of the permutation triplets  listed in the theorem that are generated by choosing the right-most path, this first term corresponds to the desired explicit formula. Using standard methods (as in Matthews \cite{Recurrence}) identical to those used in the example found in the proof of the previous theorem, it is easy to see that $m_{n}=m_{n-1}+m_{n-3}$.

\end{proof}

In the above two theorems, we have presented explicit formulas and recurrence relations for the maximal terms of TRIP-Stern sequences generated by 20 maps. Using Lemma \ref{onlye} brings this total up to $20\cdot 6=120$ maps.  

\begin{question}
What are the recurrence relations for the other 96 TRIP-Stern sequences?

\end{question}  

\begin{conjecture}
The following table presents some numerical results and conjectured recurrence relations for TRIP-Stern maximal terms corresponding to the above-mentioned 96 maps.

\end{conjecture}
\begin{center}
\scalebox{0.8}{
$\begin{array}{c|c|c|c}
\mbox{First 11 maximal terms} & \mbox{Conjectured recurrence relation} \;m_n= & (e, \tau_0, \tau_1)  & \mbox{A-number}\\
\hline \hline
1,2,3,5,7,11,16,25,36,56,81  & \mbox{unknown} & (e, 13, e),(e, 123 ,12) & \seqnum{A271485}\\
\hline
1,2,3,4,6,9,13,19,28,41,60     & m_{n-1} + m_{n-3} & (e, 13, 12) & \seqnum{A000930}\\
\hline
1,2,3,4,6,8,11,16,22,30,43   & \mbox{unknown}  & (e,13,23),(e,23,12) & \seqnum{A271486}\\
\hline
1,2,3,4,6,8,11,17,23,32,48 & \mbox{unknown}  & (e,13,132),(e,132,12) & \seqnum{A271487} \\
\hline
1,2,3,4,6,8,11,  15,21,30,41        & \mbox{unknown}  &(e,23,e), (e, 123,23) & \seqnum{A271488} \\
\hline
1,2,3,4,5,7,9,12,16,21 & m_{n-2} + m_{n-3} & \begin{array}{c} (e,23,23),(e,23,132) \\ (e,132,23),(e,132,132) \end{array}  & \seqnum{A000931}\\
\hline
1,2,3,5,8,13,21,34,55,89,144 & m_{n-1} + m_{n-2} & (e,123,e)  & \seqnum{A000045}\\
\hline
1,2,3,4,5,7,10,13,18,25,34 & \mbox{unknown} & (e,123,132),(e,132,e) & \seqnum{A271489}
\end{array}$
}
\captionof{table}{Conjectured recurrence relations for maximal terms}
\end{center}

Note that in the above table we only included maps of the form $(e,\tau_0,\tau_1)$; as before, 
Lemma \ref{onlye} brings the total up to $16\cdot 6=96$ maps.

\section{Minimal terms and positions thereof}
\label{S4}
We now investigate the positions of the minimal terms in each level for various TRIP-Stern sequences.  This is a bit easier than the analogous investigation for the maximal terms.

\begin{theorem}
\label{left}
The minimal terms $b_n$ in the TRIP-Stern sequences corresponding to the permutation triplets $(e,12,12)$, $(e,12,123)$, $(e,12,132)$, $(e,13,123)$, $(e,13,132)$, $(e,23,12)$, $(e,23,123)$, $(e,23,132)$, $(e,123,12)$, $(e,123,123)$, and $(e,123,132)$ lie on the left-most path in the corresponding TRIP-Stern tree.

\end{theorem}

\begin{proof}
For the permutation triplets $(e,13,123)$ and $(e,13,132),$ the transformation \[(x_1,x_2,x_3)\mapsto (x_1,x_2,x_3)F_0\] flips the positions of $x_2$ and $x_3$ in the triplet. Recall that application of $F_0$ corresponds to following the left-most child at each node in the TRIP-Stern tree corresponding to some permutation triplet. Therefore, as we start with $(x_1,x_2,x_3)=(1,1,1),$ it is clear that the minimal terms of the TRIP-Stern sequences corresponding to these permutation triplets lie on the left-most path in the corresponding TRIP-Stern tree.

For the  rest of the above permutation triplets, the transformation \[(x_1,x_2,x_3)\mapsto(x_1,x_2,x_3)F_0\] leaves either $x_1,$ $x_2,$ or $x_3$ fixed. As in the first part of this proof, since we start with $(x_1,x_2,x_3)=(1,1,1),$ it is clear that the minimal terms of the TRIP-Stern sequences corresponding to the above permutation triplets lie on the left-most path in the corresponding TRIP-Stern tree.

\end{proof}

\begin{theorem}
\label{right}
The minimal terms $b_n$ in the TRIP-Stern sequences corresponding to the permutation triplets $(e,e,e)$, $(e,e,12)$, $(e,e,13)$, $(e,e,23)$, $(e,13,e)$, $(e,13,13)$, $(e,13,23)$, $(e,132,e)$, $(e,132,12)$, $(e,132,13)$, and $  (e,132,23)$ lie on the right-most path in the corresponding TRIP-Stern tree.

\end{theorem}

\begin{proof}
For the permutation triplets $(e,e,12)$ and $(e,132,12)$, the transformation \[(x_1,x_2,x_3)\mapsto (x_1,x_2,x_3)F_1\] flips the positions of $x_1$ and $x_2$ in the triplet. Recall that application of $F_1$ corresponds to following the right-most child at each node in the TRIP-Stern tree corresponding to some permutation triplet. Therefore, as we start with $(x_1,x_2,x_3)=(1,1,1),$ it is clear that the minimal terms of the TRIP-Stern sequences corresponding to the above permutation triplets lie on the right-most path in the corresponding TRIP-Stern tree.

For the  rest of the above permutation triplets, the transformation \[(x_1,x_2,x_3)\mapsto(x_1,x_2,x_3)F_1\] leaves either $x_1,$ $x_2,$ or $x_3$ fixed. As in the first part of this proof, since we start with $(x_1,x_2,x_3)=(1,1,1),$ it is clear that the minimal terms of the TRIP-Stern sequences corresponding to the above permutation triplets lie on the right-most path in the corresponding TRIP-Stern tree.

\end{proof}

\begin{theorem}
\label{both}
The minimal terms $b_n$ in the TRIP-Stern sequences corresponding to the permutation triplets $(e,12,e)$, $(e,12,13)$, $(e,12,23)$, $(e,13,12)$, $(e,23,e)$, $(e,23,13)$, $(e,23,23)$, $(e,123,e)$, $(e,123,13)$,  and $(e,123,23)$ lie on both the right-most and left-most paths in the corresponding TRIP-Stern tree.

\end{theorem}

\begin{proof}

For the permutation triplet $(e,13,12)$, the transformation $(x_1,x_2,x_3)\mapsto (x_1,x_2,x_3)F_0$ flips the positions of $x_2$ and $x_3$ in the triplet and the transformation $(x_1,x_2,x_3)\mapsto (x_1,x_2,x_3)F_1$ flips the positions of $x_1$ and $x_2$. As we start with $(x_1,x_2,x_3)=(1,1,1),$ it is clear that the minimal terms of the TRIP-Stern sequences corresponding to the above permutation triplets lie on both the left- and right-most paths in the corresponding TRIP-Stern tree.

For the  rest of the above permutation triplets, both the transformations $(x_1,x_2,x_3)\mapsto(x_1,x_2,x_3)F_0$ and $(x_1,x_2,x_3)\mapsto(x_1,x_2,x_3)F_1$ leave either $x_1,$ $x_2,$ or $x_3$ fixed. Since we start with $(x_1,x_2,x_3)=(1,1,1),$ it is clear that the minimal terms of the TRIP-Stern sequences corresponding to the above permutation triplets lie on the right-most and left-most paths in the corresponding TRIP-Stern tree.
\end{proof}

\begin{corollary}

The minimal terms $b_n$ in the TRIP-Stern sequences corresponding to the permutation triplets mentioned in Theorems \ref{left}, \ref{right}, and \ref{both} all have value 1.

\end{corollary}
\begin{proof}
This follows immediately since in each case we start with $(x_1,x_2,x_3)=(1,1,1),$ and at least one of the components of this triplet gets carried over to at least one triplet in the next level of the corresponding TRIP-Stern sequence by the action of $F_0$ or $F_1,$ as outlined in the proofs of the theorems mentioned in the corollary. 
\end{proof}

In the above theorems, we have found the values and positions of the minimal terms of TRIP-Stern sequences generated by 32 maps. The results of Lemma \ref{onlye} bring this total up to $32\cdot 6=192$ maps.

\begin{question}
What are the values and positions of the minimal terms of TRIP-Stern sequences generated by the remaining 24 maps?

\end{question}

\section{Level sums} \label{sums}

The following section examines the sums of the entries in each level of a TRIP-Stern sequence, in direct analogue to the level sums found by Stern for his diatomic sequence (Lehmer \cite{Lehmer1} presents this as Property 2). As the level $n$ grows large, the ratio between the sums of the entries in successive entries approaches an algebraic number of degree at most $3$. The ratios between the first, second, and third coordinates of a given level approach ratios in the same number field. This section provides a closed form for the sums of the entries in each level.

\begin{definition}
Consider the TRIP-Stern sequence for arbitrary $(\sigma, \tau_0,\tau_1)$. Let $S_1(n)$ be the sum of the first entries of the triples in the $n^\text{th}$ level, let $S_2(n)$ be the sum of the second entries, and let $S_3(n)$ be the sum of the third entries. The sum of all entries in a given level is $S(n) = S_1(n) + S_2(n) + S_3(n)$.
\end{definition}

\begin{proposition} \label{recurrenceSi}
For each $n$, 
\[\big(S_1(n), S_2(n), S_3(n)\big) = \big(S_1(n - 1), S_2(n - 1), S_3(n - 1)\big)(F_0 + F_1).\]
\end{proposition}

\begin{proof}

The base case follows by definition. Say that we are at the $n^{\rm th}$ level. In order to generate the next level of triplets, we apply $F_0$ and $F_1$ to each triplet in the $n^{\rm th}$ level. It is clear that $S_1(n+1)$ is obtained by taking the sum of the first components of $a_m F_0+a_m F_1$ over all triplets $a_m$ in the $n^{th}$ level; similarly for $S_2(n+1)$ and $S_3(n+1)$. As a result, it is clear that \[\big(S_1(n), S_2(n), S_3(n)\big) = \big(S_1(n - 1), S_2(n - 1), S_3(n - 1)\big)(F_0 + F_1).\]

\end{proof}
\subsection{Level sums for $(e,e,e)$}

\begin{proposition}\label{limprop}
Let $\alpha$ be the real zero of $x^3 - 4x^2 + 5x - 4$.  Then
\[\lim_{n\to\infty} \frac{S_i(n)}{S_i(n-1)} = \alpha \] for $1 \leq i \leq 3$ and \[ \lim_{n\to\infty} \frac{S(n)}{S(n-1)} = \alpha.\]
\end{proposition}
\begin{proof}
The characteristic polynomial of $A_0 + A_1$ is $x^3 - 4x^2 + 5x - 4$.  This polynomial has real zero 
$\alpha \approx 2.69562,$
and two complex zeros of smaller absolute value.  Thus, as $n \to \infty$, the vector
\[\big(S_1(n), S_2(n), S_3(n)\big) = (1,1,1)(A_0+A_1)^n\]
approaches the eigenvector $\bar{\alpha}$ corresponding to the eigenvalue $\alpha$. So as $n$ approaches infinity, the vector $(1,1,1)(A_0+A_1)^n$ approaches the subspace generated by $\bar{\alpha}$. Hence,
\[\lim_{n\to\infty} \frac{S_i(n)}{S_i(n-1)} = \alpha 
\]
for $1 \leq i \leq 3$ and
\[ \lim_{n\to\infty} \frac{S(n)}{S(n-1)} = \alpha,\]
as claimed.
\end{proof}

\begin{proposition}
We have
\[ \lim_{n\to\infty} \frac{S_2(n)}{S_1(n)}= \alpha - 1 \quad\text{and}\quad \lim_{n\to\infty}\frac{S_3(n)}{S_2(n)}= \alpha - 1.\]
\end{proposition}
\begin{proof}
The actions of $A_0$ and $A_1$ yield the recurrence relation $S_1(n + 1) = S_1(n) + S_2(n)$.  Dividing by $S_1(n)$ gives 
\[\frac{S_1(n + 1)}{S_1(n)} = 1 + \frac{S_2(n)}{S_1(n)}\]
Taking the limit as $n\to \infty$ and applying Proposition \ref{limprop} yields
\[\lim_{n\to\infty}\frac{S_2(n)}{S_1(n)} = \alpha - 1.\]
Similarly, using the recurrence relation $S_2(n + 1) = S_2(n) + S_3(n)$ yields
\[\lim_{n\to\infty}\frac{S_3(n)}{S_2(n)} = \alpha - 1.\]
\end{proof}

\subsection{Level sums for arbitrary $(\sigma,\tau_0,\tau_1)$}

The properties of level sums for an arbitrary TRIP-Stern sequence are similar to those of the TRIP-Stern sequence for $(e,e,e)$. For any $(\sigma,\tau_0,\tau_1)$, it can be shown by computation that $F_0+ F_1$ has an eigenvalue of largest absolute value. In fact, this eigenvalue is an element of the interval $[2, 3]$.

\begin{proposition}
For any TRIP-Stern sequence, we have \[\lim_{n \rightarrow \infty} \frac{S_i(n)}{S_i(n -1)} = \alpha\]
$ \text{ for } 1 \leq i \leq 3$ and \[\lim_{n \rightarrow \infty} \frac{S(n)}{S(n-1)} = \alpha,\]
where $\alpha$ is the eigenvalue of $F_0+F_1$ of largest absolute value. Furthermore, $\alpha$ is an algebraic number of degree at most three, and $\alpha \in [2,3]$.
\end{proposition}

Analogous recurrence relations give relations between the limits of $S_1(n)$, $S_2(n)$, and $S_3(n)$. These are not always as clean as in the case of the TRIP-Stern sequence for $(e,e,e)$, but the ratios between these limits are contained in the number field $\mathbb{Q}(\alpha)$.

We next examine the closed forms for $S(n)$.

\begin{theorem}
The family of triangle partition maps leads to 11 distinct sequences of sums $\left(S(n)\right)_{n\geq 1}$ with recurrence relations and explicit forms as shown in the tables below. 
\end{theorem}

\begin{proof}
The proof follows by direct calculation. For each triangle partition map $T_{\sigma,\tau_0,\tau_1},$ compute a modified TRIP-stern sequence given by setting $a_1=(a,b,c)$ instead of setting $a_1=(1,1,1)$ as we had done in Section \ref{S2.5}, partitioning the sequence into levels as before. Sum the terms of each level $n$ to yield a sequence of row sums $\left(S(n)\right)_{n\geq 1}$.

If we can prove that the first $m$ terms of $\left(S(n)\right)_{n\geq 1}$ satisfy an $(m-1)$-term recurrence relation, it follows that the sequence must be generated by that recurrence relation. We have carried out this procedure for all 216 permutation triplets $(\sigma,\tau_0,\tau_1)$ to find recurrence relations for the associated row sums, from which the explicit form for the $n^{\rm th}$ term in the sequence $\left(S(n)\right)_{n\geq 1}$ was easily calculated.
The results are presented in the tables below -- indeed, the family of triangle partition maps generates only 11 distinct row sums. The first column lists the recurrence relation,  the second lists the explicit form of that recurrence relation, and  the third lists the permutation triplets whose TRIP-Stern level sums follow this relation. Note that Greek letters represent zeros of certain polynomials; see the key below. For example, $\text{Root}\left[29 x^3-87 x^2-5 x-1,1\right]\to\alpha_1$ means ``let $\alpha_1$ be the first root of $29 x^3-87 x^2-5 x-1=0$."
\end{proof}

\begin{center}
\scalebox{0.94}{
$
\begin{array}{l|l|l}
\mbox{Recurrence relation for $S(n)$} & \mbox{Explicit form for $S(n)$} & (e,\tau_0,\tau_1)\\
\hline\hline
 4 S(n-3)-5 S(n-2)+4 S(n-1) & \alpha _1 \beta _1^n+\alpha _2 \beta _2^n+\alpha _3 \beta _3^n & \begin{array}{c}
 (e,e,e),\\
 (e,123,123) \\
\end{array} \\ \hline
 2 S(n-2)+2 S(n-1) & \begin{array}{c}\frac{1}{6}( \left(9-5 \sqrt{3}\right) \left(1-\sqrt{3}\right)^n\\ +\left(1+\sqrt{3}\right)^n \left(9+5 \sqrt{3}\right))\end{array} & \begin{array}{c}
 (e,e,12), \\ 
 (e,e,123), \\ 
 (e,13,12), \\ 
 (e,13,123) \\ 
\end{array} \\ \hline
 S(n-3)-S(n-2)+3 S(n-1) & \gamma _1 \delta _1^n+\gamma _3 \delta _2^n+\gamma _2 \delta _3^n & \begin{array}{c}
 (e,e,13), \\ 
 (e,12,123) \\ 
\end{array} \\ \hline
 -S(n-3)+2 S(n-2)+2 S(n-1) & \frac{2^{-n} \left(-2 \left(3-\sqrt{5}\right)^n \left(-2+\sqrt{5}\right)+\left(3+\sqrt{5}\right)^n \left(11+5 \sqrt{5}\right)\right)}{5+\sqrt{5}} & \begin{array}{c}
 (e,e,23),\\ 
 (e,12,23), \\ 
 (e,12,132), \\ 
 (e,23,e), \\ 
 (e,23,13), \\ 
 (e,23,123), \\ 
 (e,123,23), \\ 
 (e,123,132), \\ 
 (e,132,e), \\ 
 (e,132,13) \\ 
\end{array} \\ \hline
 6 S(n-3)+2 S(n-2)+S(n-1) & \epsilon _1 \zeta _1^n+\epsilon _2 \zeta _2^n+\epsilon _3 \zeta _3^n & \begin{array}{c}
 (e,e,132), \\ 
 (e,132,123) \\ 
\end{array} \\ \hline
 5 S(n-1)-6 S(n-2) & 2^n+2\ 3^n & \begin{array}{c}
 (e,12,e), \\ 
 (e,12,13), \\ 
 (e,123,e), \\ 
 (e,123,13),\\ 
\end{array} \\ \hline
 -4 S(n-3)+S(n-2)+3 S(n-1) & \eta _1 \theta _1^n+\eta _2 \theta _2^n+\eta _3 \theta _3^n & \begin{array}{c}
 (e,12,12), \\ 
 (e,13,13) \\ 
\end{array} \\ \hline
 -S(n-3)-3 S(n-2)+4 S(n-1) & \iota _1 \kappa _1^n+\iota _2 \kappa _2^n+\iota _3 \kappa _3^n & \begin{array}{c}
 (e,13,e) \\ 
 (e,123,12) \\ 
\end{array} \\ \hline
 -6 S(n-3)+4 S(n-2)+2 S(n-1) & \lambda _1 \mu _1^n+\lambda _2 \mu _2^n+\lambda _3 \mu _3^n & \begin{array}{c}
 (e,13,23), \\ 
 (e,23,12) \\ 
\end{array} \\ \hline
 S(n-3)+4 S(n-2)+S(n-1) & \nu _1 \xi _1^n+\nu _2 \xi _2^n+\nu _3 \xi _3^n & \begin{array}{c}
 (e,13,132), \\ 
 (e,132,12) \\ 
\end{array} \\ \hline
 4 S(n-2)+S(n-1) & \begin{array}{c}\frac{1}{34} (\left(51-13 \sqrt{17}\right) \left(\frac{1}{2} \left(1-\sqrt{17}\right)\right)^n\\+\left(\frac{1}{2} \left(1+\sqrt{17}\right)\right)^n \left(51+13 \sqrt{17}\right))\end{array} & \begin{array}{c}
 (e,23,23), \\ 
 (e,23,132), \\ 
 (e,132,23), \\ 
 (e,132,132) \\ 
\end{array} \\ 
\end{array}
$
}
\captionof{table}{Level sums}
\end{center}

\begin{center}
\scalebox{0.92}{
$
\begin{array}{l|l}
\hline\hline
 \text{Root}\left[29 x^3-87 x^2-5 x-1,1\right]\to \alpha _1 & \text{Root}\left[29 x^3-87 x^2-5 x-1,2\right]\to \alpha _2 \\
 \text{Root}\left[29 x^3-87 x^2-5 x-1,3\right]\to \alpha _3 & \text{Root}\left[x^3-4 x^2+5 x-4,1\right]\to \beta _1 \\
 \text{Root}\left[x^3-4 x^2+5 x-4,2\right]\to \beta _2 & \text{Root}\left[x^3-4 x^2+5 x-4,3\right]\to \beta _3 \\
 \text{Root}\left[76 x^3-228 x^2+28 x-1,1\right]\to \gamma _1 & \text{Root}\left[76 x^3-228 x^2+28 x-1,2\right]\to \gamma _2 \\
 \text{Root}\left[76 x^3-228 x^2+28 x-1,3\right]\to \gamma _3 & \text{Root}\left[x^3-3 x^2+x-1,1\right]\to \delta _1 \\
 \text{Root}\left[x^3-3 x^2+x-1,2\right]\to \delta _2 & \text{Root}\left[x^3-3 x^2+x-1,3\right]\to \delta _3 \\
 \text{Root}\left[1176 x^3-3528 x^2-119 x-1,1\right]\to \epsilon _1 & \text{Root}\left[1176 x^3-3528 x^2-119 x-1,2\right]\to \epsilon _2 \\
 \text{Root}\left[1176 x^3-3528 x^2-119 x-1,3\right]\to \epsilon _3 & \text{Root}\left[x^3-x^2-2 x-6,1\right]\to \zeta _1 \\
 \text{Root}\left[x^3-x^2-2 x-6,2\right]\to \zeta _2 & \text{Root}\left[x^3-x^2-2 x-6,3\right]\to \zeta _3 \\
 \text{Root}\left[229 x^3-687 x^2+230 x+4,1\right]\to \eta _1 & \text{Root}\left[229 x^3-687 x^2+230 x+4,2\right]\to \eta _2 \\
 \text{Root}\left[229 x^3-687 x^2+230 x+4,3\right]\to \eta _3 & \text{Root}\left[x^3-3 x^2-x+4,1\right]\to \theta _1 \\
 \text{Root}\left[x^3-3 x^2-x+4,2\right]\to \theta _2 & \text{Root}\left[x^3-3 x^2-x+4,3\right]\to \theta _3 \\
 \text{Root}\left[49 x^3-147 x^2+35 x-1,1\right]\to \iota _1 & \text{Root}\left[49 x^3-147 x^2+35 x-1,2\right]\to \iota _2 \\
 \text{Root}\left[49 x^3-147 x^2+35 x-1,3\right]\to \iota _3 & \text{Root}\left[x^3-4 x^2+3 x+1,1\right]\to \kappa _1 \\
 \text{Root}\left[x^3-4 x^2+3 x+1,2\right]\to \kappa _2 & \text{Root}\left[x^3-4 x^2+3 x+1,3\right]\to \kappa _3 \\
 \text{Root}\left[404 x^3-1212 x^2+24 x+1,1\right]\to \lambda _1 & \text{Root}\left[404 x^3-1212 x^2+24 x+1,2\right]\to \lambda _2 \\
 \text{Root}\left[404 x^3-1212 x^2+24 x+1,3\right]\to \lambda _3 & \text{Root}\left[x^3-2 x^2-4 x+6,1\right]\to \mu _1 \\
 \text{Root}\left[x^3-2 x^2-4 x+6,2\right]\to \mu _2 & \text{Root}\left[x^3-2 x^2-4 x+6,3\right]\to \mu _3 \\
 \text{Root}\left[169 x^3-507 x^2+1,1\right]\to \nu _1 & \text{Root}\left[169 x^3-507 x^2+1,2\right]\to \nu _2 \\
 \text{Root}\left[169 x^3-507 x^2+1,3\right]\to \nu _3 & \text{Root}\left[x^3-x^2-4 x-1,1\right]\to \xi _1 \\
 \text{Root}\left[x^3-x^2-4 x-1,2\right]\to \xi _2 & \text{Root}\left[x^3-x^2-4 x-1,3\right]\to \xi _3 \\
\end{array}
$
}
\captionof{table}{Key to level sums table}
\end{center}

Note that in the above table we only included maps of the form $(e, \tau_0, \tau_1)$; as before, Lemma \ref{onlye}
brings the total up to $36*6 = 216$ maps.

\subsubsection{Six row sum recurrence relations already well-known}

It turns out that 6 of the 11 unique row sum recurrence relations or sequences -- either with the same initial terms or with different ones -- had already been placed on Sloane's encyclopedia; 
in particular, these recurrence relations and corresponding Sloane sequence numbers are as follows:
\begin{eqnarray*}
S(n)&=&2S(n-1)+2S(n-2) ~~~~(\seqnum{A080040}, \seqnum{A155020}) \\
S(n)&=&3S(n-1)-S(n-2)+S(n-3)  ~~~~(\seqnum{A200752})\\
S(n)&=&2S(n-1)+2S(n-2)-S(n-3) ~~~~(\seqnum{A061646})\\ 
S(n)&=&5S(n-1)-6S(n-2) ~~~~(\seqnum{A007689})\\
S(n)&=&4S(n-1)-3S(n-2)-S(n-3) ~~~~(\seqnum{A215404})\\
S(n)&=&S(n-1)+4S(n-2) ~~~~(\seqnum{A006131})
\end{eqnarray*}

The new sequences\footnote{We have placed these sequences on Sloane's encyclopedia.} are generated by the following recurrence relations:
\begin{eqnarray*}
S(n)&=&4S(n-1)-5S(n-2)+4S(n-3)~~~~(\seqnum{A278612})\\
S(n)&=&S(n-1)+2S(n-2)+6S(n-3)~~~~(\seqnum{A278613})\\
S(n)&=&3S(n-1)+S(n-2)-4S(n-3)~~~~(\seqnum{A278614})\\
S(n)&=&2S(n-1)+4S(n-2)-6S(n-3)~~~~(\seqnum{A278615})\\
S(n)&=&S(n-1)+4S(n-2)+S(n-3)~~~~(\seqnum{A278616})
\end{eqnarray*}

\section{Forbidden triples for $(e,e,e)$}
\label{S6}

This section explores which points in $\mathbb{Z}^3$ appear as terms in the TRIP-Stern sequence for $(e,e,e)$ and which points do not.  We call these latter points \textit{forbidden triples}.

\begin{definition} Let $S$ denote the set of points given by the TRIP-Stern sequence for $(e,e,e)$, let $P=\{(x,y,z)\in\mathbb{Z}^{3}|0<x\le y < z\}\cup(1,1,1)$ denote the set of \textit{potential entries} in $S$ and let $F=P\setminus S$ denote the set of \textit{forbidden triples}. 
\end{definition}

\begin{proposition} \label{order}
We have that $S \subseteq P$.
\end{proposition}
\begin{proof}
Suppose $(a,b,c)\in S$.  We proceed by induction on the level of $(a, b, c)$. If $(a, b,c)$ is in the second level, then $(a, b, c) = (1,1,2)$, so it satisfies $0 < a \leq b < c$. 
Suppose that $0 < a \leq b < c$ for all elements $(a, b, c)$ of the $n^{\rm th}$ level. If $(a,b,c)$ is in level $n + 1$, then $(a, b, c) = (d, e, f)A_{i_n}$ with $i_n \in \{0,1\}$ for some $(d,e,f)$ satisfying $0 < d \leq e < f$. If $i_n = 0$, then $(a, b, c) = (e, f, d + f)$. We know that $0 < e \leq f,$ and $f < d+f$ because $d$ is positive. If $i_n = 1$, then $(a,b,c) = (d,e,d+f)$. By the inductive hypothesis, $0 < d \leq e$ and $e < f < d+f$.
\end{proof}
By Proposition \ref{order}, elements of $F$ are precisely the potential entries that are not in the TRIP-Stern sequence for $(e,e,e)$.
\begin{definition} We define the \textit{inverse map} $G$ of a triple $(a,b,c)$ to be

\[
G=
\begin{cases}
(a,b,c-a), & \mbox{ if } a+b<c;\\
(c-b,a,b), & \mbox{ if } a+b \ge c  \mbox{ and } a<b, \mbox{ or } a=1=c-b.
\end{cases}
\]

\end{definition}

The map $G$ is $A_{1}^{-1}$ if $a+b<c$ and is $A_{0}^{-1}$ if $a+b \ge c$ is $A_{0}^{-1}$. However, $G$ is not defined on all points of $P$. We will show that, for any point $(a,b,c)$ in $P$ that is not $(1,1,2)$, at most one of $(a,b,c)A_{1}^{-1}$ or $(a,b,c)A_{0}^{-1}$ can lie in the TRIP-Stern sequence for $(e,e,e)$.  

\begin{proposition}\label{notP}
If $(a,b,c) \in P,$ and $(a,b,c) \neq (1,1,2)$, then either $(a,b,c)A_{0}^{-1}$ or $(a,b,c)A_{1}^{-1}$ is not in $P$.
\end{proposition}
\begin{proof}
First assume $(a,b,c)$ has $a+b \ge c$.  Then 
\[(a,b,c)A_{1}^{-1}=(a,b,c-a)\]
However, we have that $b \ge c-a$.  Then $(a,b,c)A_{1}^{-1}$ is not in $P$, unless $(a,b,c-a)=(1,1,1)$, which occurs if and only if $a=b=1, c=2$.
Now suppose $(a,b,c)$ has $a+b<c$.  Then
\[(a,b,c)A_{0}^{-1}=(c-b,a,b).\]
In order for this to lie in $P$ we need $c-b \le a < b$. In particular, $c \le b+a$, contradicting the assumption.  
Thus, either $(a,b,c)=(1,1,2)$ or at most one of $(a,b,c)A_{1}^{-1}$ or $(a,b,c)A_{0}^{-1}$ lies in $P$.
\end{proof}

\begin{corollary}
For every $X\in S$ except $(1,1,1)$, the point $X$ appears exactly twice in the TRIP-Stern sequence for $(e,e,e)$. Furthermore, $G$ maps $X$ to the unique $Y \in S$ such that either $YA_0=X$ or $YA_1=X$.
\end{corollary}
\begin{proof}
If a point $(a,b,c) \neq (1,1,1)$ lies in $S$, then by definition either $(a,b,c)A_{1}^{-1}$ or $(a,b,c)A_{0}^{-1}$ must lie in $S$.  Proposition \ref{notP} implies that only one of these can be in $P$. Thus, exactly one of these points is in $S$. This takes care of the second statement.

Because the left and right subtrees of the TRIP-Stern sequence for $(e,e,e)$ are symmetric, we need only show that each $X\in S$ appears exactly once in the set of points generated by the action of $A_{1}$ and $A_{0}$ on $(1,1,2)$.  Now suppose that up to level $n$ each entry in the TRIP-Stern sequence appears only once.  Then, in level $n+1$, each element $X$ corresponds to either $XA_{1}^{-1}$ or $XA_{0}^{-1}$ in level $n$.  By 
Proposition \ref{notP}, exactly one of $XA_{1}^{-1}$ or $XA_{0}^{-1}$ will lie in level $n$, and this is the only element that goes to $X$ under one of $A_0$ or $A_1$.  This is the unique preimage in $S$, under $A_0$ and $A_1$, that goes to $X$.  Then each entry in level $n+1$ appears for the first time, and in the level $n+1$ there are no repeated entries, completing the induction.  
\end{proof}

\begin{definition}
We define a \textit{germ} to be any element $(a,a,b) \in P$ with $ b<2a$.  We call the set of all elements generated by action of $A_0$ and $A_{1}$ on $(a,a,b)$ the tree generated by $(a,a,b)$.  
\end{definition}

Observe that the tree generated by $(1,1,1)$ is precisely the TRIP-Stern Sequence for $(e,e,e)$.  The only points for which $G$ is not defined in $P$ are germs.  Moreover, each application of $G$ to $X\in P$ decreases (strictly) the sum of the entries of $X$. Since $G$ is well-defined on all of $P$, excluding germs, after some number of applications of $G$ to an element $X$, we find a germ that generates $X$.  The following lemma, whose proof is straightforward, is needed to strengthen the result.

\begin{lemma}
\label{lemmaG}
For all $X \in P$, we have $G((X)A_1)=X$ and $G((X)A_0)=X$.
\end{lemma}

This lemma shows that there is a unique germ generating $X$ for each $X \in P$.

\begin{definition}
Let the \textit{germ of} $X$ be the value $G^n(X)$ such that $n$ is the largest integer for which $G^n(X)$ is defined.  As noted above, this $G^n(X)$ will be a germ.
\end{definition}

\begin{theorem}
Every element of $P$ lies in exactly one tree generated by a germ.  Furthermore, an element $X \in P$ lies in the tree generated by $X_0$ if and only if the germ of $X$ is $X_0$.  In particular, one can determine the germ of any given triple $(a,b,c)$ in a finite number of steps.  
\end{theorem}
\begin{proof}
If the germ of $X$ is $X_0$, then since G acts as $A_1^{-1}$ or $A_0^{-1}$ at each step, we have that $X$ can be written as $(X_0)A_{i_1}A_{i_2}\cdots A_{i_n}$ for some $i_j\in\{0,1\}$. 
Conversely, if $X$ is in the tree generated by $X_0$, then we can write $X$ as $(X_0)A_{i_1}A_{i_2}\cdots A_{i_n}$ for some $i_j\in \{0,1\}$.  But then by Lemma~\ref{lemmaG}, we have that $G^n(X)=X_0$, so the germ of $X$ is $X_0$.  
\end{proof}

\begin{corollary}
No germ except $(1,1,1)$ lies in the TRIP-Stern sequence for $(e,e,e)$.  All elements that do not lie in the TRIP-Stern Sequence are given by the action of  $A_0$ and $A_1$ on a germ other than $(1,1,1)$. 
\end{corollary}
\begin{proof}
Both $A_1^{-1}$ and $A_0^{-1}$ take elements of the form $(a,a,b)$ with $b<2a$ outside of $P$.  Then $(a,a,b)$ cannot be reached by the action of $A_1$ or $A_0$ on an element of $P$.  
\end{proof}

No doubt a similar analysis can be done for any TRIP-Stern sequence, though we do not know how clean the analogues would be.

\section{Generalized TRIP-Stern sequences}
\label{S7}

In the original definition of a TRIP-Stern sequence from Section  \ref{TRIPSternDef},  we set the first triple in the sequence to be $a_1=(1,1,1)$. However, there is nothing canonical about this choice, which  leads us to construct generalized TRIP-Stern sequences, where we set the initial triple to be $a_1=(a,b,c),$ for some  $a,b,c\in   \mathbb{R}$.

\subsection{Construction of generalized TRIP-Stern sequences}

\begin{definition}
For any permutation triplet $(\sigma,\tau_0,\tau_1)$ in $S_{3}^{3}$, the \textit{generalized TRIP-Stern sequence} of $(\sigma, \tau_0, \tau_1)$ is the unique sequence such that, for some $a,b,c\in \mathbb{R},$ $a_1 = (a,b,c)$ and, for $n\geq 1$,
\[\left\{ \!\! \begin{array}{ll}
a_{2n} &= a_n \cdot F_0 \\
a_{2n+1} &= a_n \cdot F_1
\end{array} \right.\]
The $n^{\rm th}$ level of the generalized TRIP-Stern sequence is the set of $a_m$ with $2^{n-1} \leq m < 2^n$.
\end{definition}
As for the standard TRIP-Stern sequence, each choice of $(\sigma, \tau_0, \tau_1)$ also produces some generalized TRIP-Stern sequence.

\begin{definition}
Let $(a,b,c)$ be any triple of real numbers and let   $(\sigma,\tau_0,\tau_1) \in S_3 \times S_3 \times S_3$. Let $\mathcal{T}(\sigma,\tau_0,\tau_1)_{(a,b,c)}$ denote the tree generated from $(\sigma,\tau_0,\tau_1)$ using $(a,b,c)$ as a seed. \end{definition}

\subsection{Maximum terms and positions thereof for generalized TRIP-Stern sequences}
\label{S8}

This section examines maximum terms in any given level of a generalized TRIP-Stern sequence, as well as the positions of those maximum terms within the given level. For a given seed $(a,b,c)$ and  for an 
$n$-tuple $v$ of zeros and ones, define
 \[\triangle(v)=(a,b,c)F_{i_1}F_{i_2} \cdots F_{i_n},\] which can be written as $\triangle(v)=(b_1,b_2,b_3)$. As before, let $|v|$ denote the number of entries in $v$. 

\begin{definition}
 The \textit{maximum entry} on level $n$ of a generalized TRIP-Stern sequence is $ m_n = \max_{|v| = n\:\:} \max_{i\in\{1,2,3\}}b_i(v)$.
 \end{definition}

\begin{lemma}
\label{onlye2}
Suppose $(\sigma,\tau_0,\tau_1)\in S_{3}^3,$ $\kappa\in S_3,$ and let $v=(i_1, \ldots, i_n)$. Then \[\triangle_{(\kappa \sigma, \tau_0 \kappa^{-1}, \tau_1 \kappa^{-1})} (v)=(a,b,c)\cdot \kappa \sigma A_{i_1} \tau_{i_1} \cdot \sigma A_{i_2}\tau_{i_2} \cdots \sigma A_{i_n} \tau_{i_n} \kappa^{-1}.\]
\end{lemma}
\begin{proof}
The proof uses the same technique as in Lemma \ref{onlye}.
\end{proof}

We will now examine the sequences of generalized TRIP-Stern maximal terms induced by select
triangle partition maps.

We find two broad classes of generalized TRIP-Stern sequences for paths through the corresponding trees that will locate the maximal terms at each level.  These depend to some extent on the initial seeds.   As the proofs are straightforward and similar to the earlier ones, we will omit them.

\begin{theorem}
\label{paths2}
We have 
\begin{enumerate} 
\item For any seed $(a,b,c),$ if $a\geq b\geq c>0$, the sequence of maximal terms for TRIP-Stern sequences induced by the nine maps $(e,13,123),$  $(e,e,13),$ $(e,13,13),$
$(e,23,13),$
$(e,23,123),$
$(e,123,13),$
$(e,123,123),$
$(e,132,13),$ and
$(e,132,123)$ 
 lies on the connected path of the tree $\mathcal{T}(e,\tau_0,\tau_1)_{(a,b,c)}$ made by always selecting the right edge of $\mathcal{T}(e,\tau_0,\tau_1)_{(a,b,c)}$.

\item For any seed $(a,b,c),$ if $0<a \leq b \leq c$, the sequence of maximal terms for TRIP-Stern sequences induced by the eleven  maps  $(e,e,e),$ $(e,e,12),$
$(e,e,23),$
$(e,e,123),$
$(e,e,132),$
$(e,12,e),$
$(e,12,12),$
$(e,12,13),$
$(e,12,23),$
$(e,12,123)$ and
$(e,12,132)$
 lies on the connected path through the tree $\mathcal{T}(e,\tau_0,\tau_1)_{(a,b,c)}$ made by always selecting the left edge of $\mathcal{T}(e,\tau_0,\tau_1)_{(a,b,c)}$.

\end{enumerate}

\end{theorem}

In the above theorem, under select conditions, we have accounted for sequences of generalized TRIP-Stern sequence maximal terms generated by 20 maps. Lemma \ref{onlye2} brings this total up to $20\cdot 6=120$ maps, as long as the conditions -- which simply guarantee that the components of the initial seed will have the right magnitudes to satisfy the conditions of Theorem \ref{paths2} after being acted upon by the first $\kappa$ in Lemma \ref{onlye2} -- listed in the theorem below are satisfied:

\begin{theorem}
\label{kappa}

\begin{enumerate} 

\item Let $(\sigma,\tau_0,\tau_1)\in S_{3}^{3}$ be one of the 9 permutation triplets listed in Theorem \ref{paths2}.1. For any a seed $(a,b,c),$ the sequence of maximal terms for TRIP-Stern sequences induced by maps of the form $(\kappa \sigma, \tau_0 \kappa^{-1}, \tau_1 \kappa^{-1})$ lies on the connected path through the tree $\mathcal{T}(\kappa \sigma, \tau_0 \kappa^{-1}, \tau_1 \kappa^{-1})_{(a,b,c)}$ made by  selecting the right edge of $\mathcal{T}(\kappa \sigma, \tau_0 \kappa^{-1}, \tau_1 \kappa^{-1})_{(a,b,c)},$ as long as the following conditions are satisfied:
\begin{enumerate} 
\item For $\kappa=12,$ require $b\geq a\geq c>0,$
\item for $\kappa=13,$ require $c\geq b\geq a>0,$
\item for $\kappa=23,$ require $a\geq c\geq b>0,$
\item for $\kappa=123,$ require $c\geq a\geq b>0,$ and
\item for $\kappa=132,$ require $b\geq c\geq a>0$.
\end{enumerate}

\item Let $(\sigma,\tau_0,\tau_1)\in S_{3}^{3}$ be one of the 11 permutation triplets listed in Theorem \ref{paths2}.2.
 For any seed $(a,b,c),$ the sequence of maximal terms for TRIP-Stern sequences induced by maps of the form $(\kappa \sigma, \tau_0 \kappa^{-1}, \tau_1 \kappa^{-1})$ lies on the connected path through the tree $\mathcal{T}(\kappa \sigma, \tau_0 \kappa^{-1}, \tau_1 \kappa^{-1})_{(a,b,c)}$ made by  selecting the left edge of $\mathcal{T}(\kappa \sigma, \tau_0 \kappa^{-1}, \tau_1 \kappa^{-1})_{(a,b,c)},$ as long as the following conditions are satisfied:
\begin{enumerate} 
\item For $\kappa=12,$ require $0<b\leq a\leq c,$
\item for $\kappa=13,$ require $0<c\leq b\leq a,$
\item for $\kappa=23,$ require $0<a\leq c\leq b,$
\item for $\kappa=123,$ require $0<c\leq a\leq b,$ and
\item for $\kappa=132,$ require $0<b\leq c\leq a$.
\end{enumerate} 
\end{enumerate} 
\end{theorem}
\begin{proof}
This follows immediately by the results of Theorem \ref{paths2} and remembering that $\kappa$ is simply a permutation.
\end{proof}

For any node $(r,s,t)$ in the tree $\mathcal{T}(\sigma,\tau_0,\tau_1)_{(a,b,c)},$ it is natural to consider the positions of maximal terms within the tree that would be generated using that $(r,s,t),$ and not the original $(a,b,c)$ as its seed. Clearly, the problem of characterizing these terms is equivalent to characterizing the maximal terms on the nodes below and connected with $(r,s,t)$; in this sense, the problem is one of finding a sequence of \textit{local maximal terms}.  It is straightforward to prove the analogous theorems.

\subsection{Minimal terms and positions thereof}
\label{S9}

We also have analogs for finding minimal terms for a number of generalized TRIP-Stern sequences.  As the proofs are similar to the earlier ones, we omit them.

\begin{theorem}

For any seed $(a,b,c),$ the minimal terms $b_n$ in the TRIP-Stern sequences corresponding to the permutation triplets listed below lie on the left-most path in the corresponding TRIP-Stern tree:

\begin{enumerate} 

\item
For the maps $(e,12,e), (e,12,12), (e,12,13), (e,12,23),$ $(e,12,123), (e,12,132)$ under the condition that $b<a$ and $b<c$ -- the minimal term will have value $b$ at every level.
\item
For the maps $(e,123,e),$ $(e,123,12), (e,123,13),$ $(e,123,23), (e,123,123), (e,123,132)$ under the condition that $b<a$ and $b<c$ or $c<a$ and $c<b$; correspondingly, the minimal term will have value $b$  or $c$ at every level.
\item
For the maps $(e,23,e), (e,23,12), (e,23,13), (e,23,23),$ $(e,23,123), (e,23,132),$ under the condition that $c<a$ and $c<b$ -- the minimal term will have value $c$ at every level.

\end{enumerate} 
\end{theorem}

\begin{theorem}

For any seed $(a,b,c),$ the minimal terms $b_n$ in the TRIP-Stern sequences corresponding to the permutation triplets listed below lie on the right-most path in the corresponding TRIP-Stern tree:

\begin{enumerate} 
\item
For the maps $(e,e,23), (e,12,23), (e,13,23),$ $(e,23,23), (e,123,23), (e,132,23),$ under the condition that $a<b$ and $a<c$  -- the minimal term will have value $a$ at every level.
\item
For the maps $(e,e,13), (e,12,13), (e,13,13), (e,23,13),$ $(e,123,13), (e,132,13),$ under the condition that $b<a$ and $b<c$  -- the minimal term will have value $b$ at every level.
\item
For the maps $(e,e,e),$ $(e,12,e),$ $(e,13,e), (e,23,e), (e,123,e), (e,132,e),$ under the condition that $a<b$ and $a<c$ XOR $b<a$ and $b<c$; correspondingly, the minimal term will have value $a$ or $b$ at every level.
\end{enumerate} 

\end{theorem}

In the above two theorems, we have found the positions and values of the minimal terms of generalized TRIP-Stern sequences generated by 27 maps under certain conditions on the initial seed. Note that certain maps appear in both the ``left" and ``right" lists. Using the results of Lemma \ref{onlye2} and imposing conditions analogous to those used in Theorem \ref{kappa} for maximal terms -- which simply guarantee that the components of the initial seed will have the right magnitudes to satisfy the conditions of the above two theorems after being acted upon by the first $\kappa$ in Lemma \ref{onlye2} --  brings this total up to $27\cdot 6=162$ maps.

\subsection{Level sums for generalized TRIP-Stern sequences}
\label{S10}

It is natural, as was done with standard TRIP-Stern sequences, to consider level sums for generalized TRIP Stern sequences. We take an identical approach to that for standard TRIP-Stern sequences, though the explicit forms are more complex, as one would expect.

\begin{theorem}
The family of triangle partition maps leads to 11 distinct sequences of sums $(S(n))_{n\geq 1}$ with recurrence relations and explicit forms as shown in the tables below. 
\end{theorem}

\begin{proof}
The proof follows by direct calculation. For each triangle partition map $T_{\sigma,\tau_0,\tau_1},$ compute a generalized TRIP-stern sequence given by setting $a_1=(a,b,c)$ instead of setting $a_1=(1,1,1)$ as we had done in Section \ref{S2.5}, partitioning the sequence into levels as before. Sum the terms of each level $n$ to yield a sequence of row sums $\left(S(n)\right)_{n\geq 1}$.

If we can prove that the first $m$ terms of $\left(S(n)\right)_{n\geq 1}$ satisfy an $(m-1)$-term recurrence relation, it follows that the sequence must be generated by that recurrence relation. We have carried out this procedure for all 216 permutation triplets $(\sigma,\tau_0,\tau_1)$ to find recurrence relations for the associated row sums, from which the explicit form for the $n^{\rm th}$ term in the sequence $\left(S(n)\right)_{n\geq 1}$ was easily calculated.

The results are presented in the tables below -- indeed, the family of triangle partition maps generates only 11 distinct row sums. The first column lists the recurrence relation, the second lists the explicit form of that recurrence relation and the third lists the permutation triplets whose TRIP-Stern level sums follow this relation. Note that Greek letters represent zeros of certain polynomials; see the key below.  For example, $\text{Root}\left[x^3 - 4 x^2 + 5 x - 4, 1\right]\to\alpha_1$ means ``let $\alpha_1$ be the first root of $x^3 - 4 x^2 + 5 x - 4=0$."
\end{proof}

\begin{center}
\scalebox{0.85}{
$
\begin{array}{c|c|c}
\mbox{Recurrence relation for $S(n)$} & \mbox{Explicit form for $S(n)$} & (e,\tau_0,\tau_1)\\
\hline\hline
 4 S(n-3)-5 S(n-2)+4 S(n-1) & \begin{array}{c}\left(c \zeta _1+b \eta _1\right) \alpha _1^n+a \left(\beta _1 \alpha _1^n+\alpha _2^n \beta _2+\alpha _3^n \beta _3\right)\\+\alpha _3^n \left(b \gamma _3+c \delta _2\right) \epsilon _2+\alpha _2^n \left(c \zeta _3+b \eta _2\right)\end{array} & 
\begin{array}{c}
 (e,e,e), \\
 (e,123,123) \\
\end{array}
 \\ \hline
 2 S(n-2)+2 S(n-1) & \begin{array}{c}\frac{1}{6} (\left(1-\sqrt{3}\right)^n (\left(3-2 \sqrt{3}\right) a\\-\left(-3+\sqrt{3}\right) b+\left(3-2 \sqrt{3}\right) c)\\+\left(1+\sqrt{3}\right)^n (\left(3+2 \sqrt{3}\right) a\\+\left(3+\sqrt{3}\right) b+\left(3+2 \sqrt{3}\right)
   c))\end{array} & 
\begin{array}{c}
 (e,e,12), \\
 (e,e,123),\\
 (e,13,12), \\
 (e,13,123) \\
\end{array}
 \\ \hline
 S(n-3)-S(n-2)+3 S(n-1) & \begin{array}{c}a \iota _1 \theta _1^n+b \kappa _1 \theta _1^n+c \mu _1 \theta _1^n+a \theta _2^n \iota _2\\+b \theta _2^n \kappa _3+\theta _3^n \left(a \iota _3+b \kappa _2+c \mu _2\right)+c \theta _2^n \mu _3 \end{array}& 
\begin{array}{c}
 (e,e,13), \\
 (e,12,123) \\
\end{array}
 \\ \hline
 -S(n-3)+2 S(n-2)+2 S(n-1) & \begin{array}{c}\frac{2^{-n}}{5 \left(5+\sqrt{5}\right)} ((-2)^n \left(5+\sqrt{5}\right) (b-c)\\ +\left(3-\sqrt{5}\right)^n (-5 \left(-1+\sqrt{5}\right) a\\+\left(5-3 \sqrt{5}\right) b-2 \left(-5+\sqrt{5}\right) c)\\+\left(3+\sqrt{5}\right)^n (10 \left(2+\sqrt{5}\right) a\\+\left(15+7
   \sqrt{5}\right) b+4 \left(5+2 \sqrt{5}\right) c))\end{array} & 
\begin{array}{c}
 (e,e,23),\\
 (e,12,23), \\
 (e,12,132), \\
 (e,23,e), \\
 (e,23,13), \\
 (e,23,123), \\
 (e,123,23), \\
 (e,123,132), \\
 (e,132,e), \\
 (e,132,13) \\
\end{array}
 \\ \hline
 6 S(n-3)+2 S(n-2)+S(n-1) & \begin{array}{c} c \xi _1 \nu _1^n+b o_1 \nu _1^n+a \rho _1 \nu _1^n+c \nu _3^n \xi _2+b \nu _3^n o_3+\\ \nu _2^n \left(c \xi _3+b o_2+a \rho _2\right)+a \nu _3^n \rho _3 \end{array}& 
\begin{array}{c}
 (e,e,132), \\
 (e,132,123) \\
\end{array}
 \\ \hline
 5 S(n-1)-6 S(n-2) & 2^n b+3^n (a+c) & 
\begin{array}{c}
 (e,12,e), \\
 (e,12,13), \\
 (e,123,e), \\
 (e,123,13) \\
\end{array}
 \\ \hline
 -4 S(n-3)+S(n-2)+3 S(n-1) & \begin{array}{c}c \tau _2 \sigma _1^n+a \upsilon _1 \sigma _1^n+b \phi _1 \sigma _1^n+c \sigma _2^n \tau _1+a \sigma _2^n \upsilon _2+ \\b \sigma _2^n \phi _2+\sigma _3^n \left(c \tau _3+a \upsilon _3+b \phi _3\right)\end{array} & 
\begin{array}{c}
 (e,12,12), \\
 (e,13,13) \\
\end{array}
 \\ \hline
 -S(n-3)-3 S(n-2)+4 S(n-1) &\begin{array}{c} c \psi _1 \chi _1^n+a \omega _2 \chi _1^n+b \digamma _1 \chi _1^n+c \chi _2^n \psi _2+a \chi _2^n \omega _1\\+b \chi _2^n \digamma _2+\chi _3^n \left(c \psi _3+a \omega _3+b \digamma _3\right)\end{array} &
\begin{array}{c}
 (e,13,e), \\
 (e,123,12) \\
\end{array}
 \\ \hline
-6 S(n-3)+4 S(n-2)+2 S(n-1) & \begin{array}{c} a \varepsilon _2 \Pi_1^n+c \vartheta _1 \Pi_1^n+b \varsigma_1 \Pi_1^n+a \varepsilon _1 \Pi_2^n+c \vartheta _2 \Pi_2^n\\ +b \varsigma_2 \Pi_2^n+\left(a \varepsilon _3+c \vartheta
   _3+b \varsigma_3\right) \Pi_3^n \end{array}& 
\begin{array}{c}
 (e,13,23), \\
 (e,23,12) \\
\end{array}
 \\ \hline
 S(n-3)+4 S(n-2)+S(n-1) & \begin{array}{c}a \varpi _2 \varkappa _1^n+c \varrho _1 \varkappa _1^n+b \varphi _2 \varkappa _1^n+a \varkappa _2^n \varpi _1+c \varkappa _2^n \varrho _2\\ +b \varkappa _2^n \varphi _1+\varkappa _3^n \left(a \varpi _3+c \varrho _3+b \varphi _3\right)\end{array} & 
\begin{array}{c}
 (e,13,132), \\
 (e,132,12) \\
\end{array}
 \\ \hline
 4 S(n-2)+S(n-1) & \begin{array}{c}\frac{1}{34} (\left(\frac{1}{2} \left(1-\sqrt{17}\right)\right)^n (\left(17-5 \sqrt{17}\right) a\\ +\left(17-3 \sqrt{17}\right) b+\left(17-5 \sqrt{17}\right) c)\\ +\left(\frac{1}{2} \left(1+\sqrt{17}\right)\right)^n (\left(17+5 \sqrt{17}\right)
   a\\+\left(17+3 \sqrt{17}\right) b+\left(17+5 \sqrt{17}\right) c))\end{array} &
\begin{array}{c}
 (e,23,23), \\
 (e,23,132),\\
 (e,132,23), \\
 (e,132,132) \\
\end{array}
\end{array}$
}
\captionof{table}{Level sums}
\end{center}

\begin{center}
\scalebox{0.85}{
$
\begin{array}{l|l}
\hline\hline
 \text{Root}\left[x^3-4 x^2+5 x-4,1\right]\to \alpha _1 & \text{Root}\left[x^3-4 x^2+5 x-4,2\right]\to \alpha _2 \\
 \text{Root}\left[x^3-4 x^2+5 x-4,3\right]\to \alpha _3 & \text{Root}\left[58 x^3-58 x^2-17 x-2,1\right]\to \beta _1 \\
 \text{Root}\left[58 x^3-58 x^2-17 x-2,2\right]\to \beta _2 & \text{Root}\left[58 x^3-58 x^2-17 x-2,3\right]\to \beta _3 \\
 \text{Root}\left[x^3+4 x^2+x+2,3\right]\to \gamma _3 & \text{Root}\left[x^3+5 x^2-3 x+1,2\right]\to \delta _2 \\
 \text{Root}\left[116 x^3+x+1,2\right]\to \epsilon _2 & \text{Root}\left[116 x^3-116 x^2-7 x-1,1\right]\to \zeta _1 \\
 \text{Root}\left[116 x^3-116 x^2-7 x-1,3\right]\to \zeta _3 & \text{Root}\left[116 x^3-116 x^2+25 x-2,1\right]\to \eta _1 \\
 \text{Root}\left[116 x^3-116 x^2+25 x-2,2\right]\to \eta _2 & \text{Root}\left[x^3-3 x^2+x-1,1\right]\to \theta _1 \\
 \text{Root}\left[x^3-3 x^2+x-1,2\right]\to \theta _2 & \text{Root}\left[x^3-3 x^2+x-1,3\right]\to \theta _3 \\
 \text{Root}\left[19 x^3-19 x^2-3 x-1,1\right]\to \iota _1 & \text{Root}\left[19 x^3-19 x^2-3 x-1,2\right]\to \iota _2 \\
 \text{Root}\left[19 x^3-19 x^2-3 x-1,3\right]\to \iota _3 & \text{Root}\left[38 x^3-38 x^2+10 x-1,1\right]\to \kappa _1 \\
 \text{Root}\left[38 x^3-38 x^2+10 x-1,2\right]\to \kappa _2 & \text{Root}\left[38 x^3-38 x^2+10 x-1,3\right]\to \kappa _3 \\
 \text{Root}\left[x^3-3 x^2+x-1,1\right]\to \lambda _1 & \text{Root}\left[x^3-3 x^2+x-1,2\right]\to \lambda _2 \\
 \text{Root}\left[x^3-3 x^2+x-1,3\right]\to \lambda _3 & \text{Root}\left[76 x^3-76 x^2-2 x-1,1\right]\to \mu _1 \\
 \text{Root}\left[76 x^3-76 x^2-2 x-1,2\right]\to \mu _2 & \text{Root}\left[76 x^3-76 x^2-2 x-1,3\right]\to \mu _3 \\
 \text{Root}\left[x^3-x^2-2 x-6,1\right]\to \nu _1 & \text{Root}\left[x^3-x^2-2 x-6,2\right]\to \nu _2 \\
 \text{Root}\left[x^3-x^2-2 x-6,3\right]\to \nu _3 & \text{Root}\left[147 x^3-147 x^2-7 x-1,1\right]\to \xi _1 \\
 \text{Root}\left[147 x^3-147 x^2-7 x-1,2\right]\to \xi _2 & \text{Root}\left[147 x^3-147 x^2-7 x-1,3\right]\to \xi _3 \\
 \text{Root}\left[588 x^3-588 x^2+77 x-4,1\right]\to o_1 & \text{Root}\left[588 x^3-588 x^2+77 x-4,2\right]\to o_2 \\
 \text{Root}\left[588 x^3-588 x^2+77 x-4,3\right]\to o_3 & \text{Root}\left[1176 x^3-1176 x^2-161 x-6,1\right]\to \rho _1 \\
 \text{Root}\left[1176 x^3-1176 x^2-161 x-6,2\right]\to \rho _2 & \text{Root}\left[1176 x^3-1176 x^2-161 x-6,3\right]\to \rho _3 \\
 \text{Root}\left[x^3-3 x^2-x+4,1\right]\to \sigma _1 & \text{Root}\left[x^3-3 x^2-x+4,2\right]\to \sigma _2 \\
 \text{Root}\left[x^3-3 x^2-x+4,3\right]\to \sigma _3 & \text{Root}\left[229 x^3-229 x^2-33 x+1,1\right]\to \tau _1 \\
 \text{Root}\left[229 x^3-229 x^2-33 x+1,2\right]\to \tau _2 & \text{Root}\left[229 x^3-229 x^2-33 x+1,3\right]\to \tau _3 \\
 \text{Root}\left[229 x^3-229 x^2+5 x+2,1\right]\to \upsilon _1 & \text{Root}\left[229 x^3-229 x^2+5 x+2,2\right]\to \upsilon _2 \\
 \text{Root}\left[229 x^3-229 x^2+5 x+2,3\right]\to \upsilon _3 & \text{Root}\left[229 x^3-229 x^2+61 x-2,1\right]\to \phi _1 \\
 \text{Root}\left[229 x^3-229 x^2+61 x-2,2\right]\to \phi _2 & \text{Root}\left[229 x^3-229 x^2+61 x-2,3\right]\to \phi _3 \\
 \text{Root}\left[x^3-4 x^2+3 x+1,1\right]\to \chi _1 & \text{Root}\left[x^3-4 x^2+3 x+1,2\right]\to \chi _2 \\
 \text{Root}\left[x^3-4 x^2+3 x+1,3\right]\to \chi _3 & \text{Root}\left[49 x^3-49 x^2+1,1\right]\to \psi _1 \\
 \text{Root}\left[49 x^3-49 x^2+1,2\right]\to \psi _2 & \text{Root}\left[49 x^3-49 x^2+1,3\right]\to \psi _3 \\
 \text{Root}\left[49 x^3-49 x^2-14 x+1,1\right]\to \omega _1 & \text{Root}\left[49 x^3-49 x^2-14 x+1,2\right]\to \omega _2 \\
 \text{Root}\left[49 x^3-49 x^2-14 x+1,3\right]\to \omega _3 & \text{Root}\left[49 x^3-49 x^2+14 x-1,1\right]\to \digamma _1 \\
 \text{Root}\left[49 x^3-49 x^2+14 x-1,2\right]\to \digamma _2 & \text{Root}\left[49 x^3-49 x^2+14 x-1,3\right]\to \digamma _3 \\
 \text{Root}\left[x^3-2 x^2-4 x+6,1\right]\to \Pi_1 & \text{Root}\left[x^3-2 x^2-4 x+6,2\right]\to \Pi_2 \\
 \text{Root}\left[x^3-2 x^2-4 x+6,3\right]\to \Pi_3 & \text{Root}\left[101 x^3-101 x^2+19 x-1,1\right]\to \varsigma_1 \\
 \text{Root}\left[101 x^3-101 x^2+19 x-1,2\right]\to \varsigma_2 & \text{Root}\left[101 x^3-101 x^2+19 x-1,3\right]\to \varsigma_3 \\
 \text{Root}\left[202 x^3-202 x^2-42 x-1,1\right]\to \varepsilon _1 & \text{Root}\left[202 x^3-202 x^2-42 x-1,2\right]\to \varepsilon _2 \\
 \text{Root}\left[202 x^3-202 x^2-42 x-1,3\right]\to \varepsilon _3 & \text{Root}\left[404 x^3-404 x^2-14 x+3,1\right]\to \vartheta _1 \\
 \text{Root}\left[404 x^3-404 x^2-14 x+3,2\right]\to \vartheta _2 & \text{Root}\left[404 x^3-404 x^2-14 x+3,3\right]\to \vartheta _3 \\
 \text{Root}\left[x^3-x^2-4 x-1,1\right]\to \varkappa _1 & \text{Root}\left[x^3-x^2-4 x-1,2\right]\to \varkappa _2 \\
 \text{Root}\left[x^3-x^2-4 x-1,3\right]\to \varkappa _3 & \text{Root}\left[169 x^3-169 x^2-26 x+1,1\right]\to \varpi _1 \\
 \text{Root}\left[169 x^3-169 x^2-26 x+1,2\right]\to \varpi _2 & \text{Root}\left[169 x^3-169 x^2-26 x+1,3\right]\to \varpi _3 \\
 \text{Root}\left[169 x^3-169 x^2-13 x+5,1\right]\to \varrho _1 & \text{Root}\left[169 x^3-169 x^2-13 x+5,2\right]\to \varrho _2 \\
 \text{Root}\left[169 x^3-169 x^2-13 x+5,3\right]\to \varrho _3 & \text{Root}\left[169 x^3-169 x^2+26 x-1,1\right]\to \varphi _1 \\
 \text{Root}\left[169 x^3-169 x^2+26 x-1,2\right]\to \varphi _2 & \text{Root}\left[169 x^3-169 x^2+26 x-1,3\right]\to \varphi _3 \\
\end{array}
$
}
\captionof{table}{Key to level sums table}
\end{center}

\section{Conclusion}
\label{Conclusion} 

This paper has used a collection of multidimensional continued fractions to construct a family of sequences called TRIP-Stern sequences. These sequences reflect the properties of the multidimensional continued fractions from which they are generated. We have studied the sequences of maximal and minimal terms -- and positions thereof. We have also characterized the sums of levels and examined restrictions on terms appearing in a given TRIP-Stern sequence. We found that several of the level sum sequences or corresponding recurrence relations are well-known. Lastly, we introduced generalized TRIP-Stern sequences and proved several analogous results. 

We will conclude with a few unanswered questions: Do recurrence relations for row maxima and their locations exist in general? We ask this because such relations could not always be found. What are the forbidden triples corresponding to $(\sigma,\tau_0,\tau_1)\in S_{3}^{3}$ other than $(e,e,e)$? Do the distributions of terms in the TRIP-Stern sequences have any interesting properties? The terms of the TRIP-Stern sequences are the denominators of the convergents of the corresponding multidimensional continued fractions. Can these sequences reveal anything about approximation properties of multidimensional continued fractions? As discussed in Dasaratha et al.\  \cite{SMALL11q1}, many known multidimensional continued fractions are combinations of our family of $216$ maps; as a result, it may be of interest to construct analogous sequences using select combination maps.

There are  polynomial analogs of Stern's diatomic sequence (as in the work of Dilcher and Stolarsky \cite{Dilcher-Stokarsky07, Dilcher-Stokarsky09},  of Coons \cite{Coons10}, of Dilcher and Ericksen \cite{Dilcher-Ericksen14}, of Klav\v{z}ar, Milutinovi\'{c}
and Petr \cite{Klavzar-Milutinovic-Petr07}, of Ulas  \cite{Ulas-Ulas11, Ulas12}, of Vargas \cite{Vargas12}, of   Bundschuh \cite{Bundschuh12}, of Bundschun and  V\"{a}\"{a}n\"{a}nen \cite{Bundschun-Vaananen13}     and of Allouche and Mend\`{e}s France \cite{Allouche-Mendes France12}).  What are the polynomial analogs for TRIP-Stern sequences?

 In essence, in this paper we start with a triple of numbers $v=(a,b,c)$ and two $3\times 3$ matrices $A$ and $B$ and then examine the concatenation of the triples
\[vA, vB, vAA, vAB, vBA, vBB, vAAA, vAAB, vABA,  \ldots.\]
Since there are 216 different triangle partition algorithms, we have 216 different pairs of $3\times 3$ matrices.  Naively then, we would expect for there to be 216 different types of sequences, or, 216 different stories.  As we have seen, this is not the case.  We have found clear patterns and classes among the 216 different TRIP-Stern sequences.  The real question is why these patterns exist.  Further, do the TRIP-Stern sequences that share, say common sequences of maximum terms, have common number theoretic properties?  These questions strike us as hard.

\section{Acknowledgments}
We thank L. Pedersen and the referee for useful comments and the National Science Foundation for their support of this research via grant DMS-0850577.

\bigskip
\hrule
\bigskip

\noindent 2010 {\it Mathematics Subject Classification}:
Primary 11B83; Secondary 11A55, 11J70, 40A99.

\noindent \emph{Keywords: }
Stern's diatomic sequence,
multidimensional continued fraction.

\bigskip
\hrule
\bigskip

\noindent (Concerned with sequences
\seqnum {A000045}, \seqnum {A000930}, \seqnum {A000931}, \seqnum
{A006131}, \seqnum {A007689}, \seqnum {A061646}, \\ \seqnum {A080040},
\seqnum {A155020}, \seqnum {A200752}, \seqnum {A215404}, \seqnum 
{A271485}, \seqnum {A271486}, \seqnum {A271487}, \seqnum {A271488}, 
\\ \seqnum {A271489}, \seqnum {A278612}, \seqnum {A278613}, \seqnum 
{A278614}, \seqnum {A278615}, and \seqnum {A278616}.)
\bigskip
\hrule
\bigskip

\vspace*{+.1in}
\noindent
Received 
Published in {\it } .

\bigskip
\hrule
\bigskip

\noindent
Return to
\htmladdnormallink{Journal of Integer Sequences home page}{http://www.cs.uwaterloo.ca/journals/JIS/}.
\vskip .1in

\end{document}